\documentclass[9pt]{amsart}
\usepackage{amssymb}
\usepackage[mathscr]{euscript}
\usepackage[all]{xy}
\usepackage[dvips]{graphics}
\usepackage{amsmath,amssymb,epsfig}

\newtheorem{theorem}{Theorem}
\newtheorem{lemma}[theorem]{Lemma}

\newtheorem{proposition}[theorem]{Proposition}
\newtheorem{corollary}[theorem]{Corollary}

\begin{document}
\title{On the Goussarov-Polyak-Viro Finite-Type Invariants and the Virtualization Move}
\author{Micah W. Chrisman}
\begin{abstract} In this paper, it is shown that there are no nonconstant Goussarov-Polyak-Viro finite-type invariants that are invariant under the virtualization move.  As an immediate corollary, we obtain the theorem of \cite{myfirst} which states none of the Birman coefficients of the Jones-Kauffman polynomial are of GPV finite type. 
\end{abstract}
\maketitle
\section{Introduction}
\footnote{This is a preprint of a paper submitted for consideration for publication to the Journal of Knot Theory and its Ramifications.}In the realm of virtual knots, there are two notions of finite-type invariant.  One method, proposed by Kauffman in \cite{virtkauff} is closely related to the Vassiliev theory for classical knots.  The class of finite-type invariants, due to Goussarov-Polyak-Viro \cite{MR1763963} (abbreviated GPV in what follows), are themselves all Kauffman finite-type invariants.  However, they are defined by a different filtration in the set of virtual knots.

The two classes of invariants possess similar structures.  In the classical case, the Vassiliev finite-type invariants have weight systems arising from semisimple Lie algebras.In the virtual case, the GPV finite-type invariants have weight systems arising from the Lie bialgebras associated to semisimple Lie algebras \cite{Polyak}.

On the other hand, not all Kauffman finite-type invariants are of GPV finite-type.  This was first observed for small orders by Kauffman in \cite{virtkauff}.  The result was sharpened in \cite{myfirst} to show that while all of the Birman coefficients $v_n$ are Kauffman finite-type of order $\le n$, none are of GPV finite-type of order $\le m$ for any $m$.

A key player in the proof of this result is the virtualization move. In a small neighborhood of the crossing the virtualization move is given by:
\begin{figure}[h]
$\begin{array}{c} \scalebox{.25}{\psfig{figure=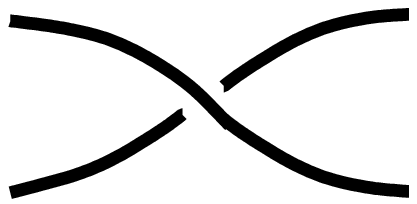}} \end{array} \leftrightarrow \begin{array}{c} \scalebox{.25}{\psfig{figure=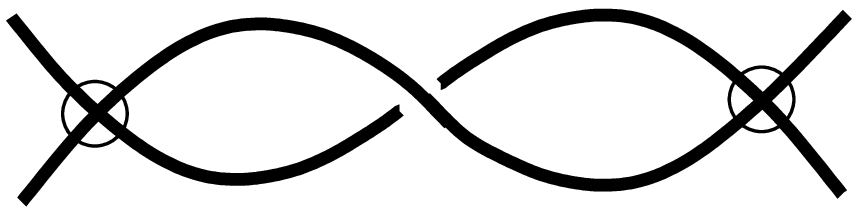}} \end{array}
$
\caption{The Virtualization Move} \label{virtmove}
\end{figure}

While this move is not a virtual isotopy move in and of itself, there are numerous virtual knot invariants which are invariant under the virtualization move. The virtualization move was discovered by Kauffman in \cite{virtkauff} in his investigation of the Jones polynomial. The Jones-Kauffman polynomial is invariant under the virtualization move. The refined theorem in \cite{myfirst} uses this invariance together with a twist sequence argument to show that the Birman coefficients are not of GPV finite type.

This leads to the question which is the subject of this paper: Are there any nonconstant GPV finite-type invariants which are invariant under the virtualization move? We answer this question in the negative with the following theorem.

\begin{theorem}\label{bigthm} If $v$ is a GPV finite-type invariant of virtual knots or long virtual knots which is invariant under the virtualization move, then $v$ is constant.  More specifically, if $v$ is a nonconstant GPV finite-type invariant, then there are knots $K$ and $K'$ such that $K$ and $K'$ are obtained from one another by a virtualization move and $v(K) \ne v(K')$.
\end{theorem} 

The proof of this theorem is surprisingly elementary.  It uses only a few basic facts about the Polyak algebra, the algebra of arrow diagrams, and module theory. We present here a proof with all details laid bare.

There is a well-known conjecture about the virtualization move called the Virtualization Conjecture \cite{MR2191949}.  It states that if $K_1$ and $K_2$ are classical knots which are obtained from one another by a sequence of virtualization moves and generalized Reidemeister moves, then $K_1$ and $K_2$ are classically isotopic.  The fact that there are many invariants which are unchanged by the virtualization move is positive evidence for this conjecture. 
In light of the conjecture, Theorem \ref{bigthm} is quite curious indeed.

This paper is organized as follows.    In the remainder of Section 1, we review the construction of the GPV finite type invariants.  The goal of Section 2 is to prove Theorem \ref{bigthm}.

The author would like to express his deep gratitude to Vassily Manturov who encouraged this investigation and patiently responded to the author's numerous (and characteristically) bad ideas.  The author would also like to thank the Monmouth University Mathematics Department for their generous financial support of this research.
\subsection{Review of GPV Finite-Type Invariants}  In this section, we review the two notions of finite-type invariants for virtual knots. First recall how one obtains a Gauss diagram $D$ of an oriented virtual knot diagram $K$. Traverse the circle in specified direction.  Every time one arrives at a classical crossing, mark a corresponding point on a copy of $S^1$ (called the Wilson loop).  If two marked points on $S^1$ correspond to the same classical crossing, connect them with an arrow.  The arrowhead is incident to the arc on $S^1$ which corresponds to the underpassing arc on $K$.  Moreover, we attach a sign to each arrow which gives the orientation of the crossing:
\[
\begin{array}{cc}
\scalebox{.15}{\psfig{figure=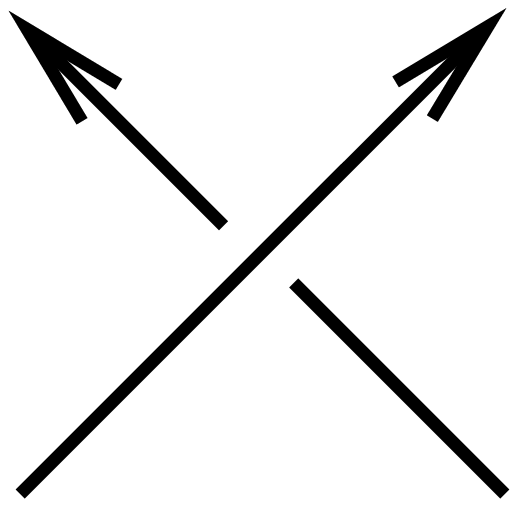}} & \begin{array}{c} \scalebox{.15}{\psfig{figure=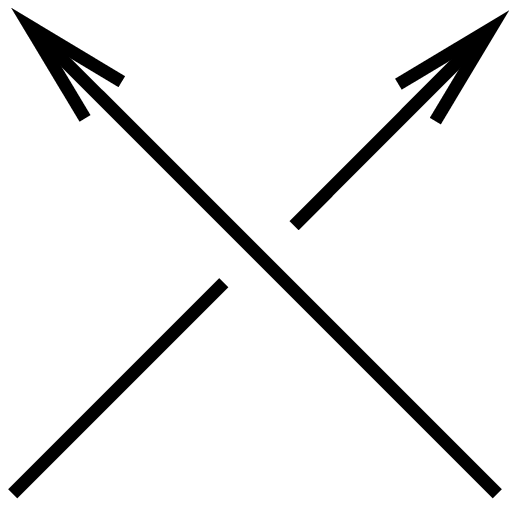}}\end{array} \\
\text{Arrow carries sign: }\oplus & \text{Arrow carries sign: } \ominus
\end{array}
\]
It is known that if two knots have the same Gauss diagram, then they are virtually isotopic via a sequence of virtual moves (see \cite{MR1763963}). The same construction works just as well for long virtual knots.  Instead of the Wilson loop, we use the Wilson line.  It is just a copy of $\mathbb{R}$.         

The virtualization move is given in Figure \ref{virtmove}.  Notice that for any of the ways in which the arcs might be directed, the local crossing number remains the same.  However, the crossing changes from over to under or vice versa.  The affect on the Gauss diagram is easy to describe:
\[
\varepsilon \begin{array}{c} \scalebox{.15}{\psfig{figure=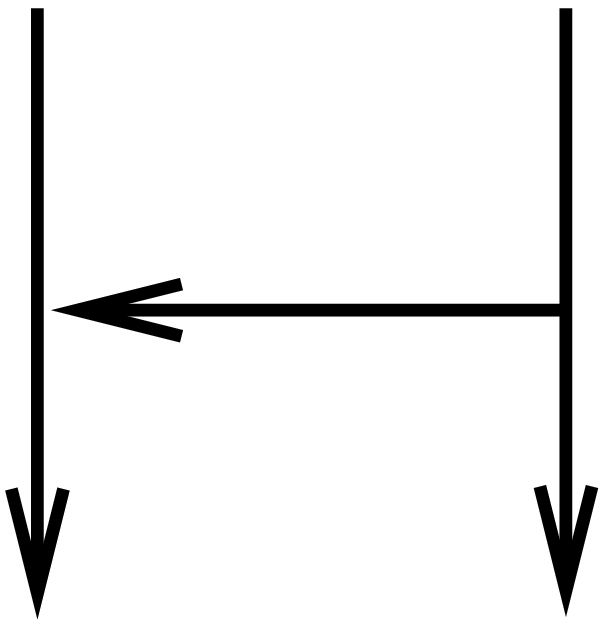}} \end{array} \leftrightarrow \varepsilon \begin{array}{c} \scalebox{.15}{\psfig{figure=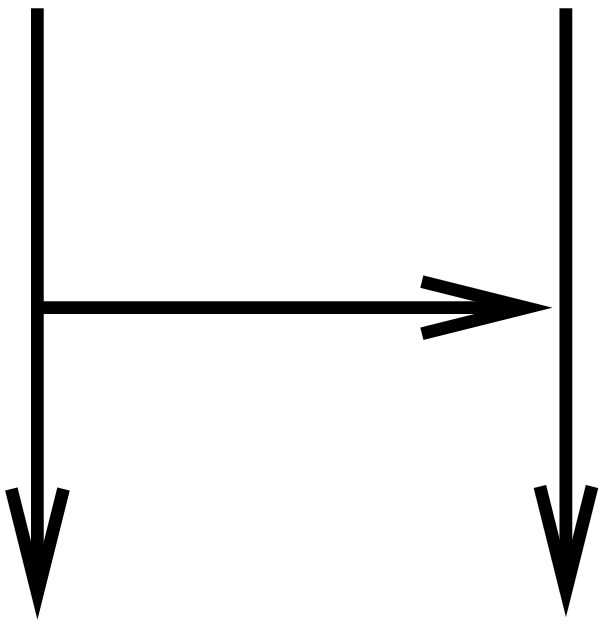}} \end{array}
\]
We now proceed to the defintion of Kauffman finite-type invariants.  First, the class of virtual knots is extended to the class of four valent graphs modulo rigid vertex isotopy (see \cite{virtkauff} for precise definition).  Vertices take the place of the singular crossings that appear in the Vassiliev theory of classical knots.  Any virtual knot invariant can be extended to these graphs by applying the following relation:
\[
v\left(\begin{array}{c} \scalebox{.15}{\psfig{figure=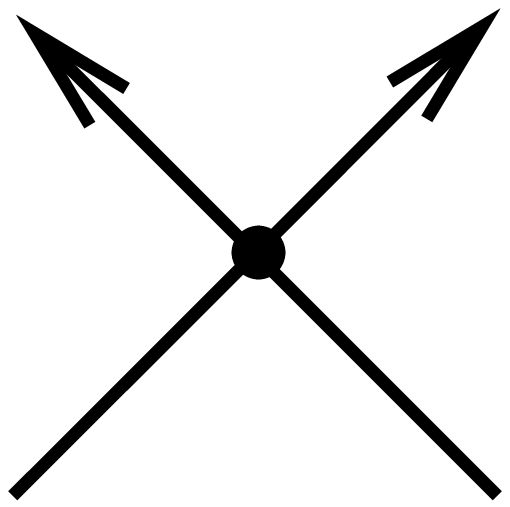}} \end{array} \right)=v\left( \begin{array}{c} \scalebox{.15}{\psfig{figure=orienrightcross.eps}} \end{array}\right)-v\left(\begin{array}{c} \scalebox{.15}{\psfig{figure=orienleftcross.eps}}\end{array}\right)
\]
A knot invariant is said to be of Kauffman finite type $\le n$ if $v(K_{\dagger})=0$ for all graphs $K_{\dagger}$ with greater than $n$ vertices.  The coefficient of $x^n$ in the power series expansion of the Birman substitution (i.e. $A \to e^x$) of the Jones-Kauffman polynomial is known to be of Kauffman finite-type $\le n$ (see \cite{virtkauff}).

The second kind of finite-type invariant of virtual knots is due to Goussarov,Polyak, and Viro \cite{MR1763963}. Virtual knots are extended to knots having \emph{semivirtual} crossings.  Any virtual knot invariant can be extended to this larger class by applying the relation:
\[
v\left(\begin{array}{c} \scalebox{.15}{\psfig{figure=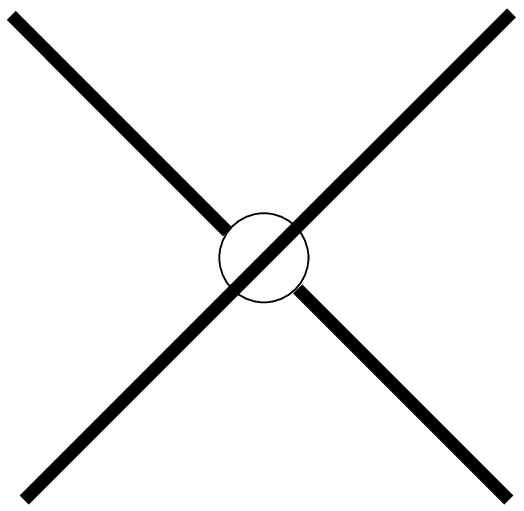}} \end{array} \right)=v\left( \begin{array}{c} \scalebox{.15}{\psfig{figure=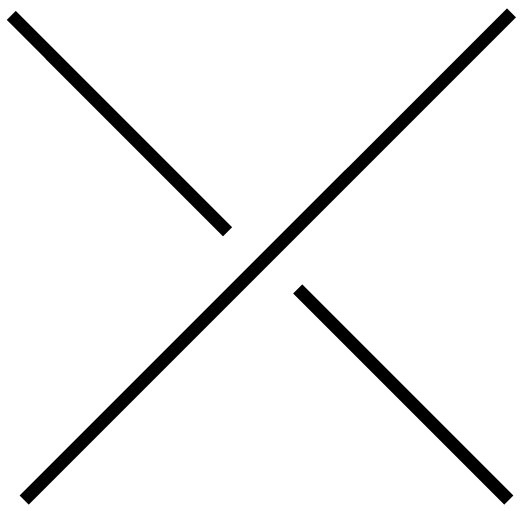}} \end{array}\right)-v\left(\begin{array}{c} \scalebox{.15}{\psfig{figure=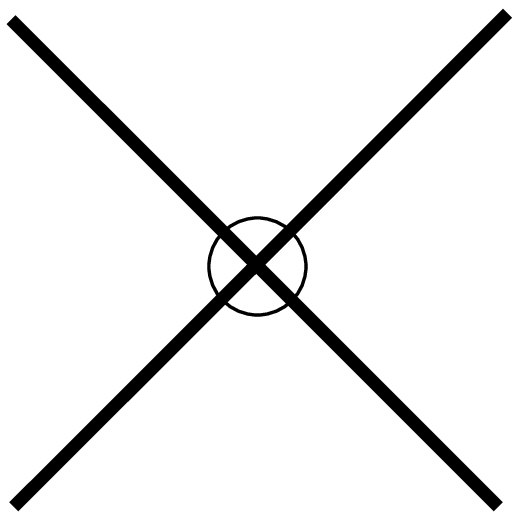}}\end{array}\right)
\]
Semivirtual crossings are represented in Gauss diagrams by dashed arrows.  Schematically, we have:
\[
\begin{array}{c} \scalebox{.17}{\psfig{figure=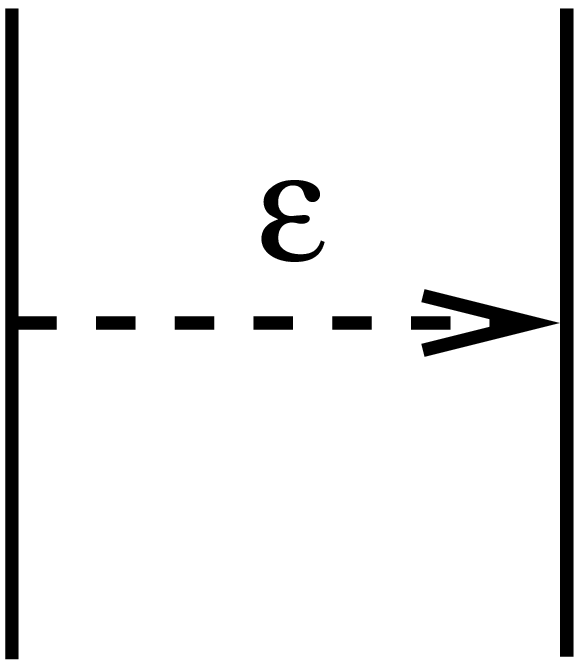}} \\ \end{array} =  \begin{array}{c} \scalebox{.17}{\psfig{figure=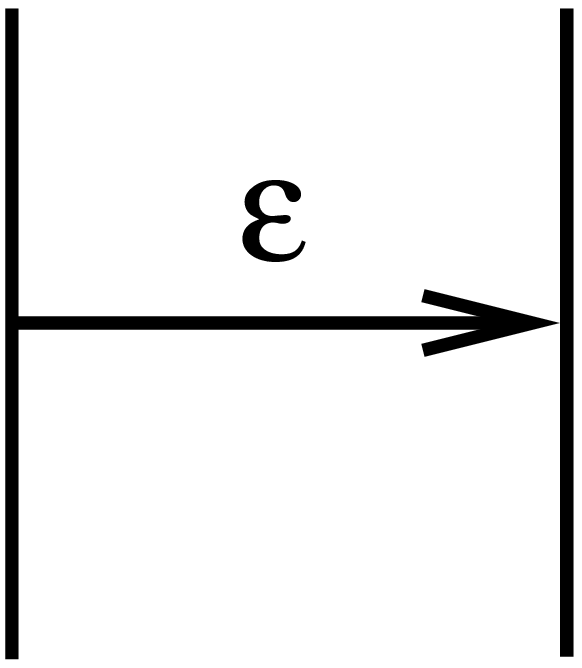}} \\ \end{array}-\begin{array}{c} \scalebox{.17}{\psfig{figure=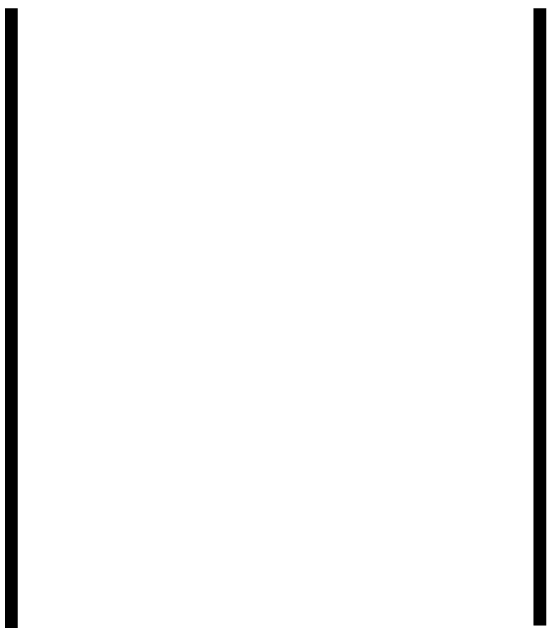}} \\\end{array}
\]

A virtual knot invariant is said to be of GPV finite-type $\le n$ if $v(K_{\circ})=0$ for all diagrams $K_{\circ}$ with greater than $n$ semivirtual crossings. There is a simple and elegant algorithm for constructing all GPV finite-type invariants of order $\le n$ that is bounded only by how much computing power one has readily available .  Moreover, every GPV finite-type invariant of order $\le n$ is of Kauffman finite-type $\le n$.  However, the aforementioned coefficients obtained from the Jones-Kauffman polynomial are not of GPV finite-type $\le n$ for \emph{any} $n$ (see \cite{myfirst}). 

We now recall the construction of the rational GPV finite-type invariants found in \cite{MR1763963}. Let $\mathscr{D}$ denote the set of all Gauss diagrams of virtual knots. A subdiagram of $D \in \mathscr{D}$ is a Gauss diagram consisting of some subset of the edges of $D$. Denote by $\mathscr{A}$ the free abelian group generated by \emph{dashed} Gauss diagrams.  The elements of $A$ are just Gauss diagrams with all arrows drawn dashed.  Define a map $i:\mathbb{Z}[\mathscr{D}] \to \mathscr{A}$ to be the map which makes all the arrows of a Gauss diagram dashed.  Define $I_{\text{GPV}}:\mathbb{Z}[\mathscr{D}] \to \mathscr{A}$ by:
\[
I_{\text{GPV}}(D)=\sum_{D' \subset D} i(D')
\]
where the sum is over all subdiagrams of $D$.  The Polyak algebra is the quotient of $\mathscr{A}$ by the submodule $\Delta\mathscr{P}=\left<\Delta \text{PI},\Delta \text{PII},\Delta \text{PIII} \right>$ generated by the relations in Figure \ref{polyakrels}.  The Polyak algebra is denoted $\mathscr{P}$.
\begin{figure}[h]  
\[
\underline{\Delta \text{PI}:}\,\,\, \begin{array}{c}\scalebox{.15}{\psfig{figure=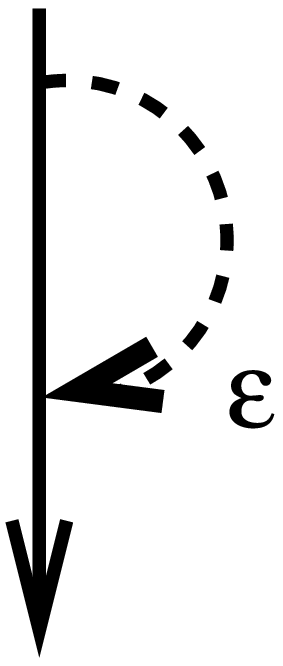}} \end{array} =0, \underline{\Delta \text{PII}:}\,\,\,\begin{array}{c}\scalebox{.15}{\psfig{figure=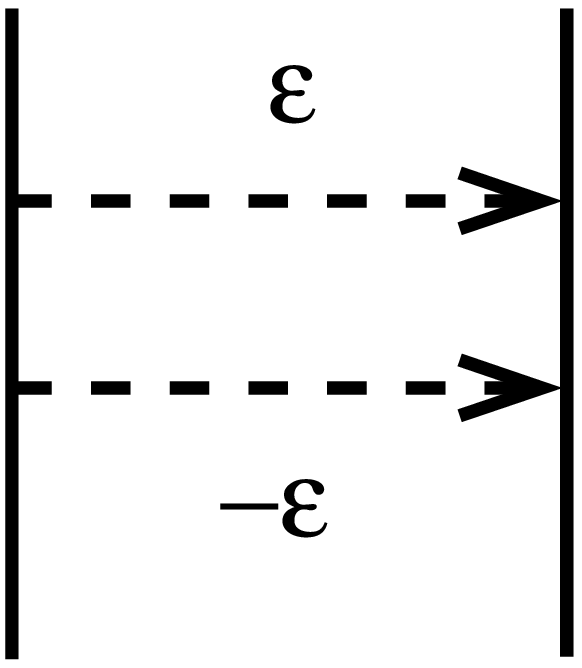}} \end{array}+\begin{array}{c}\scalebox{.15}{\psfig{figure=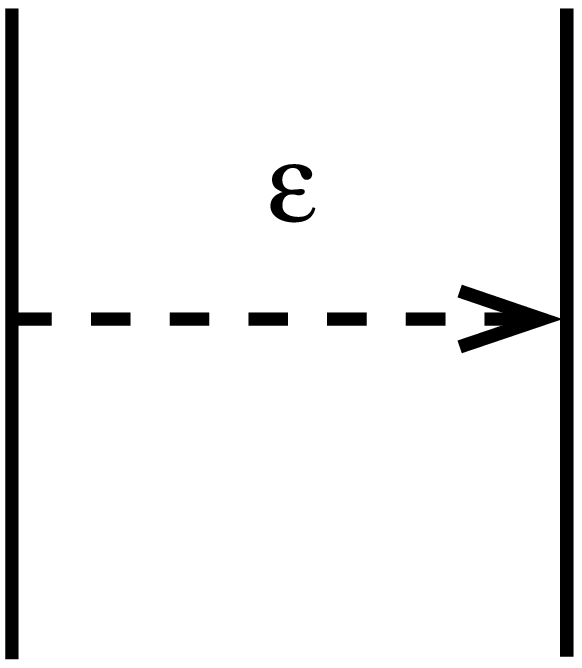}} \end{array}+\begin{array}{c}\scalebox{.15}{\psfig{figure=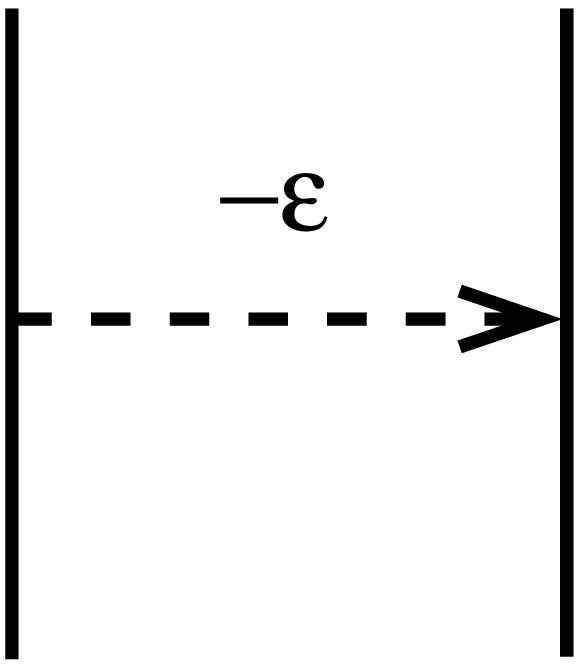}} \end{array}=0,
\]
\begin{eqnarray*}
\underline{\Delta \text{PIII}:}\,\,\,\begin{array}{c}\scalebox{.15}{\psfig{figure=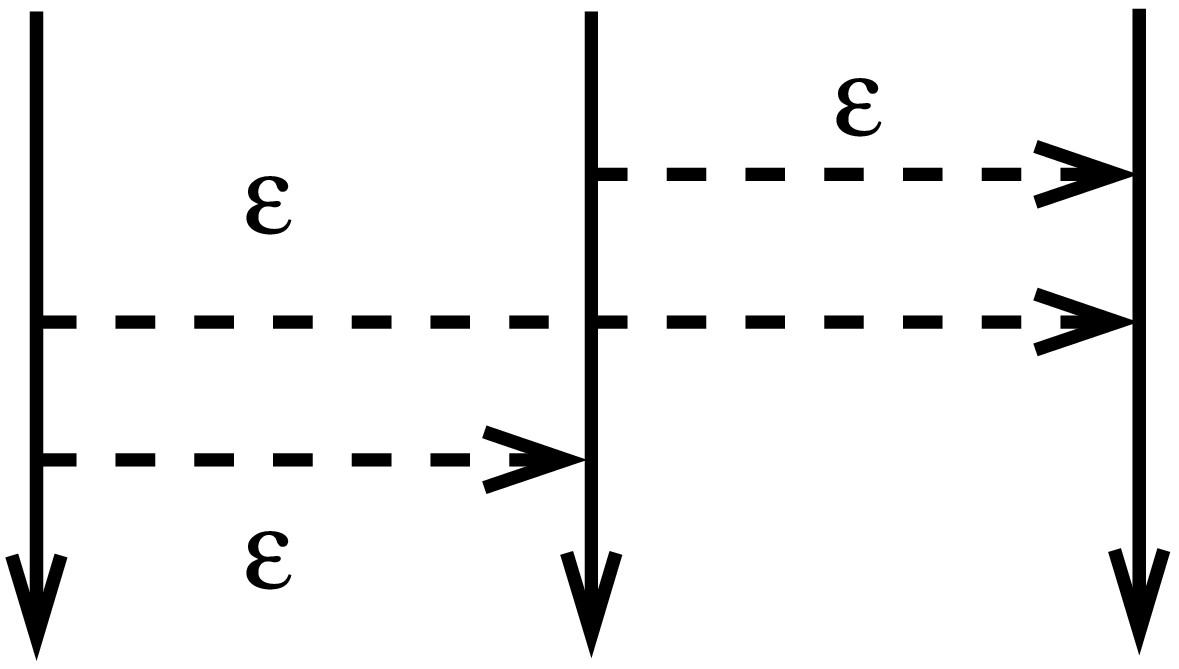}} \end{array}+\begin{array}{c}\scalebox{.15}{\psfig{figure=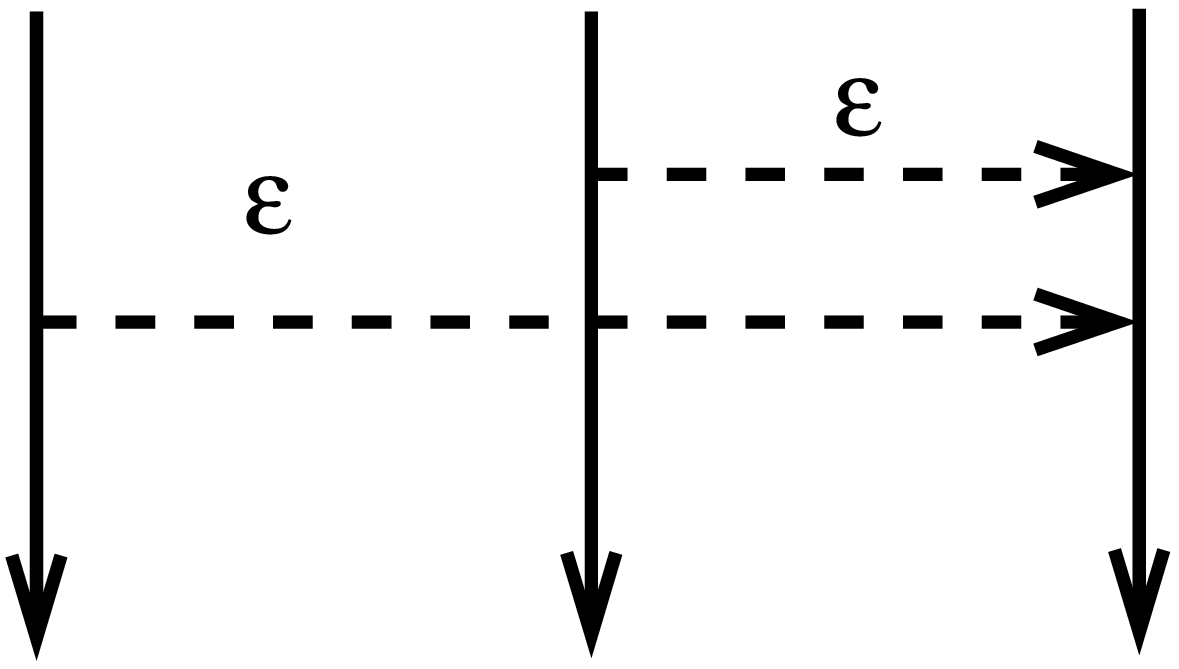}} \end{array}+\begin{array}{c}\scalebox{.15}{\psfig{figure=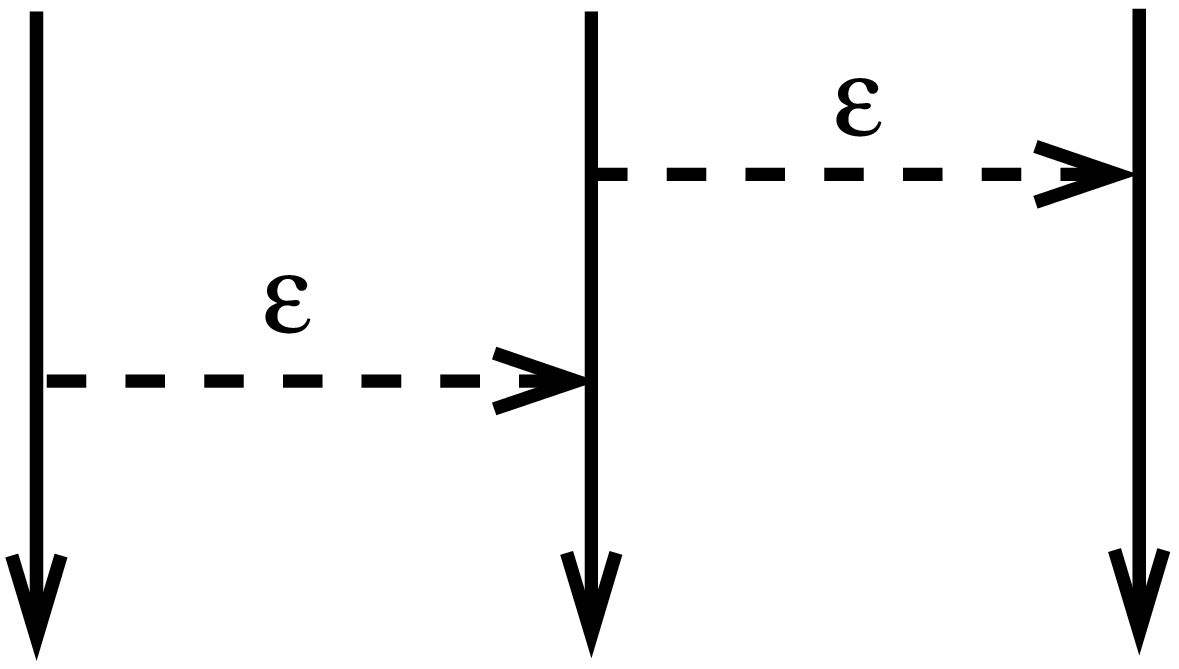}} \end{array}+\begin{array}{c}\scalebox{.15}{\psfig{figure=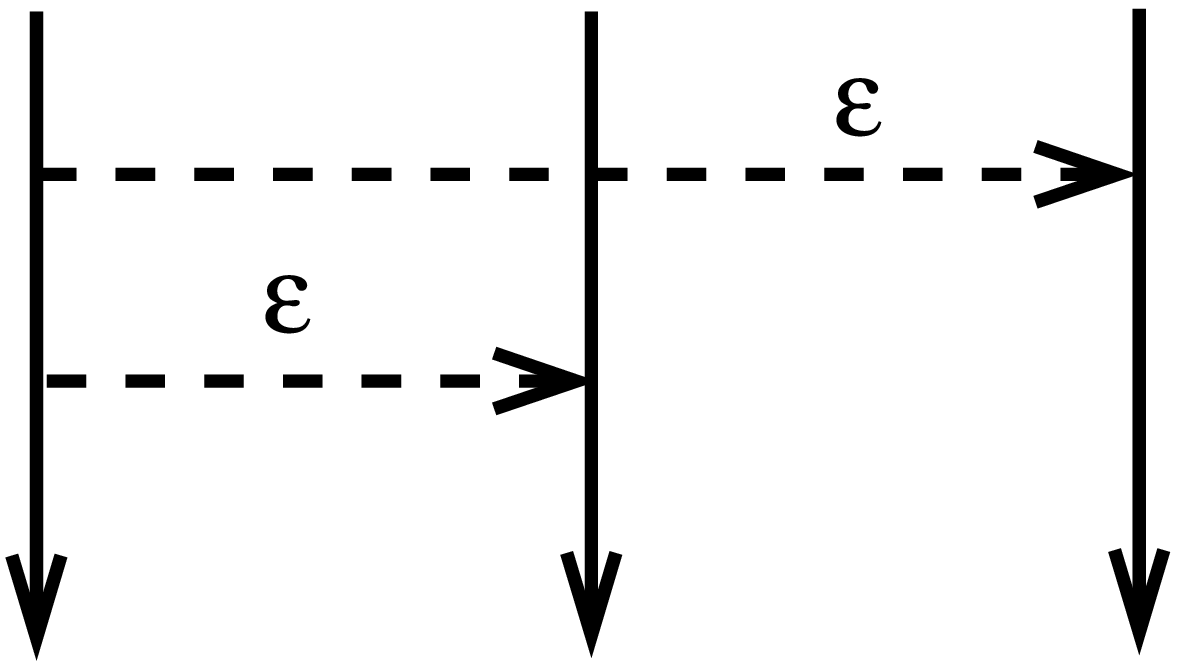}} \end{array} &=& \\ \begin{array}{c}\scalebox{.15}{\psfig{figure=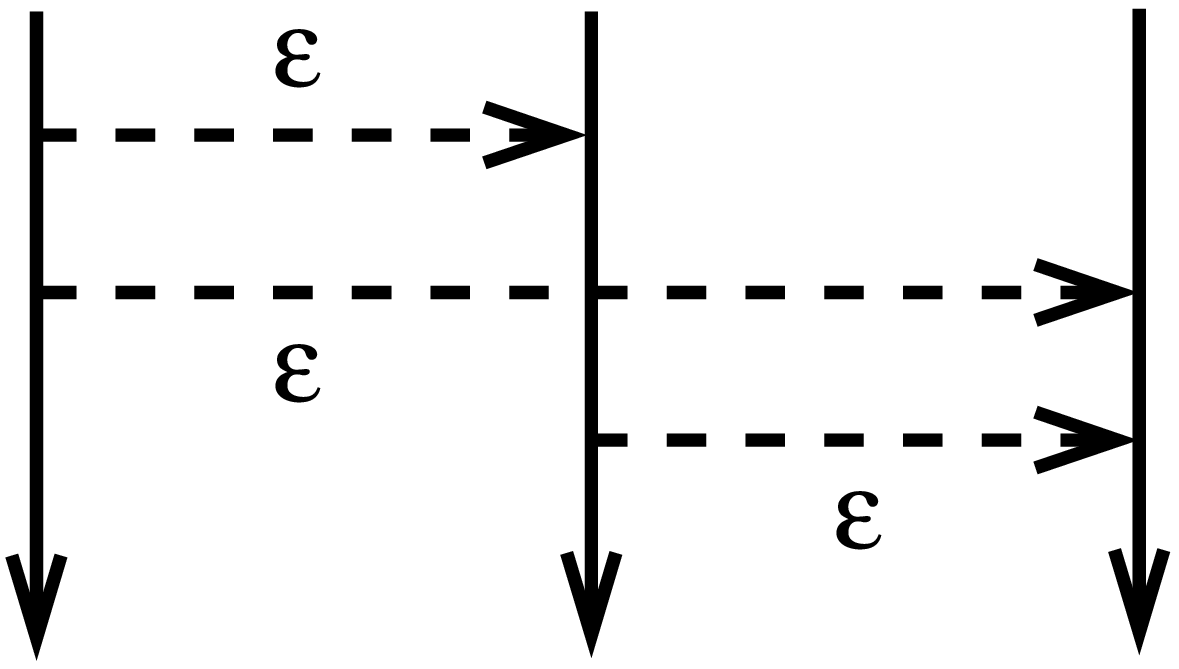}} \end{array}+\begin{array}{c}\scalebox{.15}{\psfig{figure=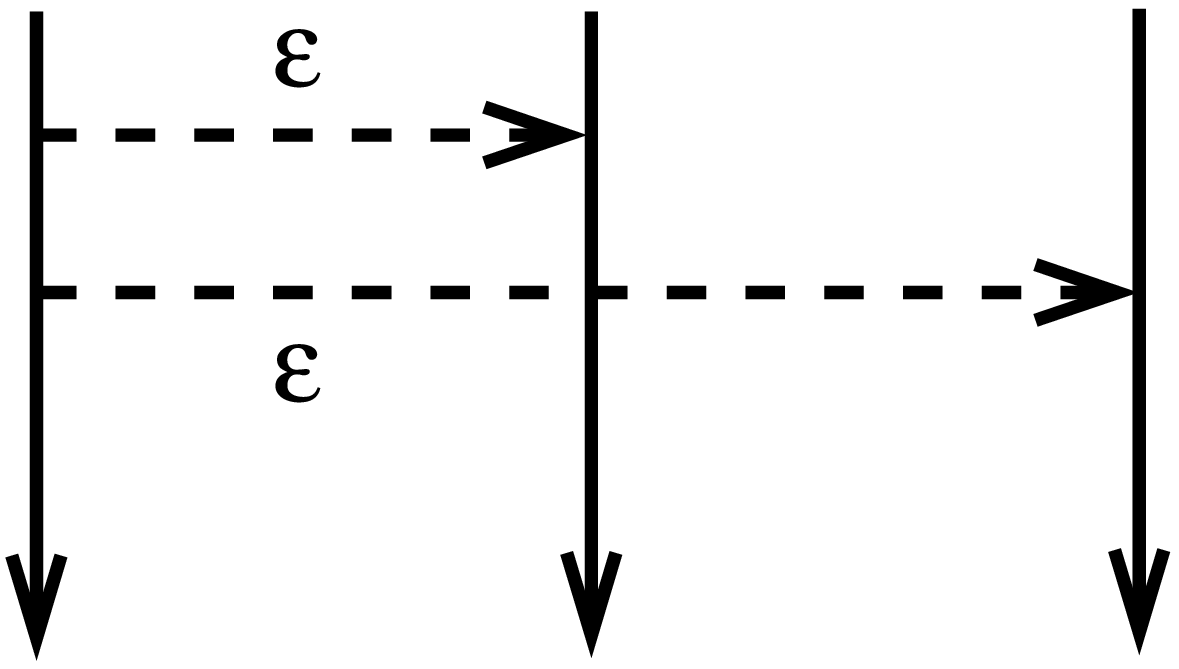}} \end{array}+\begin{array}{c}\scalebox{.15}{\psfig{figure=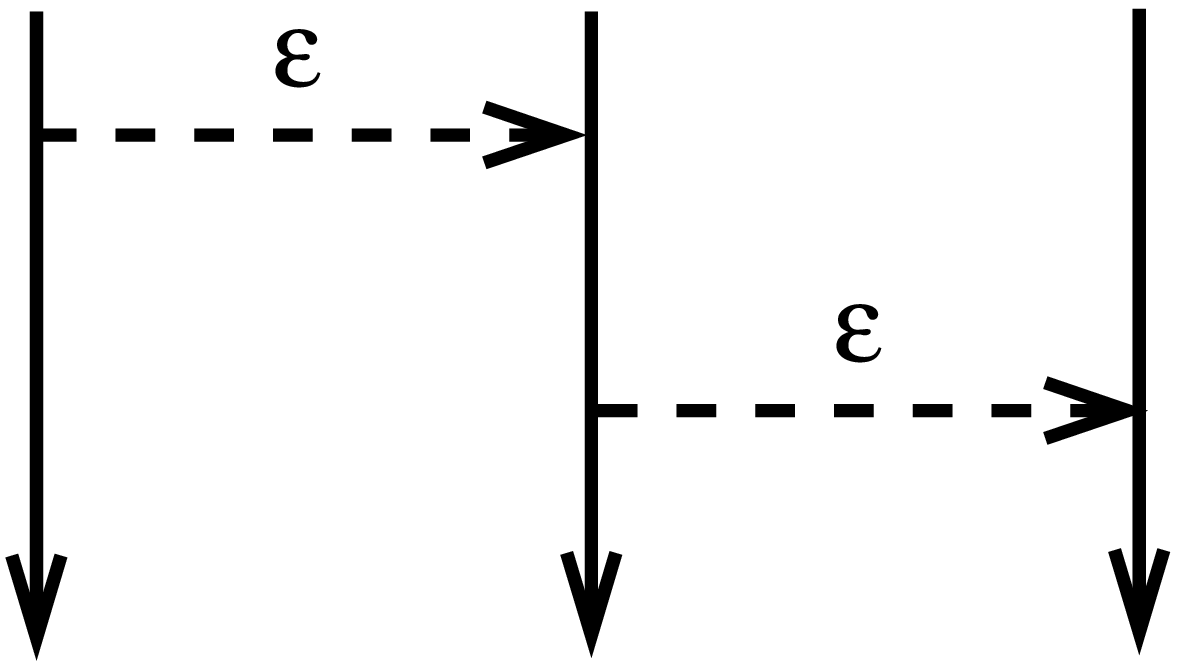}} \end{array}+\begin{array}{c}\scalebox{.15}{\psfig{figure=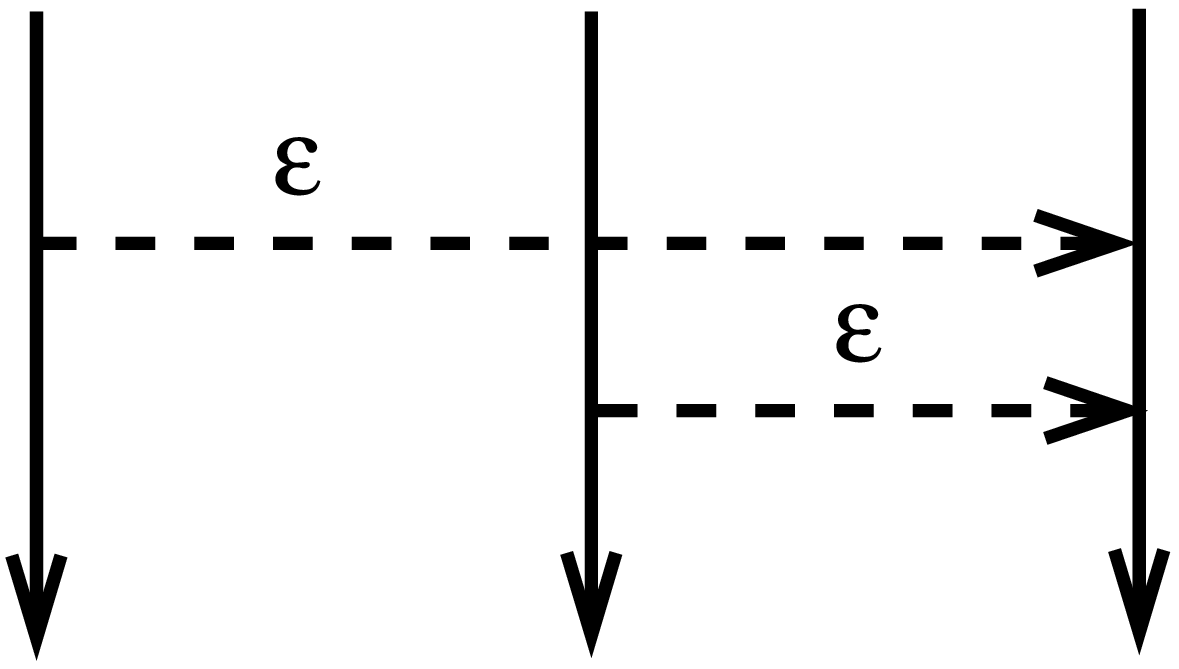}} \end{array} & & \\
\end{eqnarray*}
\caption{Polyak Relations} \label{polyakrels}
\end{figure}

Denote by $\mathscr{K}$ the set of virtual knots (where the elements are virtual isotopy classes of knots).  The following theorem shows that the theory of virtual knots is entirely encoded in the Polyak algebra.
\begin{theorem}[Goussarov, Polyak, Viro, \cite{MR1763963}] \label{iisom} The map $I_{\text{GPV}}:\mathbb{Z}[\mathscr{D}] \to \mathscr{A}$ is an isomorphism.  The inverse can be defined explicitly:
\[
I_{\text{GPV}}^{-1}(A)=\sum_{A'\subset A} (-1)^{|A-A'|} i^{-1}(A)
\]
Here, $|A- A'|$ means the number of arrows in $A$ that are not in $A'$.  Furthermore, if $D \in \mathbb{Z}[\mathscr{D}]$ has dashed arrows, then every element in the sum defining $I(D)$ also has every dashed arrow of $D$.  Finally, the map extends to an isomorphism of the quotient algebras $I_{\text{GPV}}:\mathbb{Z}[\mathscr{K}] \to \mathscr{P}$.
\end{theorem}

Let $A_n$ denote submodule of $\mathscr{A}$ generated by those diagrams having more than $n$ arrows.  Define $\mathscr{P}_n=\mathscr{A}/(A_n + \Delta \mathscr{P})$. Let $\varphi_n:\mathscr{P} \to \mathscr{P}_n$ denote the natural projection onto the quotient.  The following important theorem characterizes all rational valued GPV finite-type invariants.
\begin{theorem} [Goussarov, Polyak, Viro, \cite{MR1763963}] \label{GPVuni} The map $(I_{\text{GPV}})_n:\mathbb{Z}[\mathscr{K}] \to \mathscr{P} \to \mathscr{P}_n$ is universal in the sense that if $G$ is any abelian group, and $v$ is a GPV finite-type invariant of order $\le n$, then there is a map $v':\mathscr{P}_n \to G$ such that the following diagram commutes:
\[
\xymatrix{\mathbb{Z}[\mathscr{K}] \ar[r]^v \ar[d]_{I_{GPV}} & G \\
\mathscr{P} \ar[r]_{\varphi_n} \ar[ur]_{v I^{-1}_{GPV}} & \mathscr{P}_n \ar@{-->}[u]_{v'} \\
}
\]
In particular, the vector space of rational valued invariants of type $\le n$ is finite dimensional and can be identified with $\text{Hom}_{\mathbb{Z}}(\mathscr{P}_n,\mathbb{Q})$.
\end{theorem}
\section{Proof of Main Theorem}
\subsection{A model of virtualization invariant knot invariants} Let $\mathscr{C}$ denote the set of \emph{signed chord} diagrams. These are chord diagrams in the usual sense which have the additional structure of a sign at each chord: $\oplus$ or $\ominus$. Let $C_n$ denote the free abelian group generated by set of signed chord diagrams possessing $>n$ chords. The relations for $\mathbb{Z}[\mathscr{C}]$ are as follows:
\begin{figure}[h]  
\[
\underline{\Delta \text{RI}:}\,\,\, \begin{array}{c}\scalebox{.15}{\psfig{figure=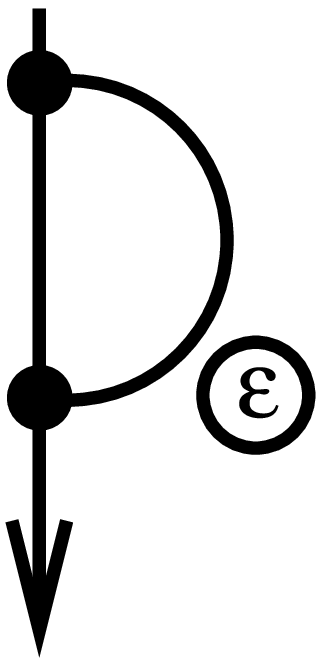}} \end{array} =0, \underline{\Delta \text{RII}:}\,\,\, \begin{array}{c}\scalebox{.15}{\psfig{figure=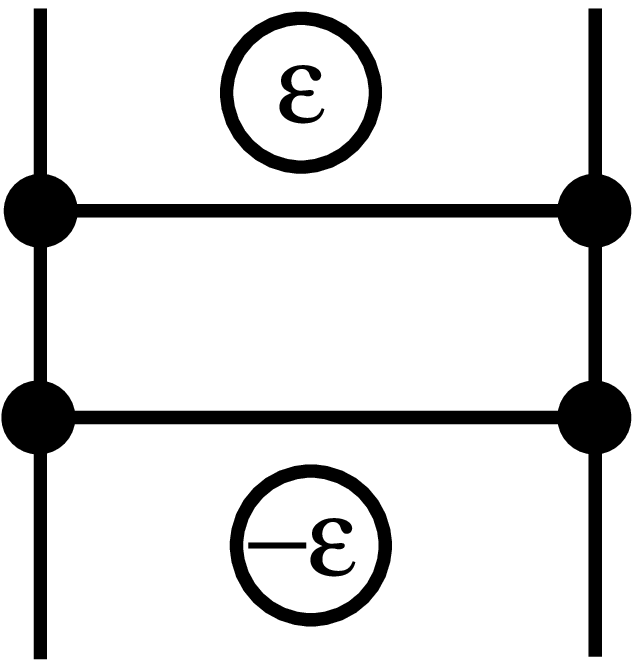}} \end{array}+\begin{array}{c}\scalebox{.15}{\psfig{figure=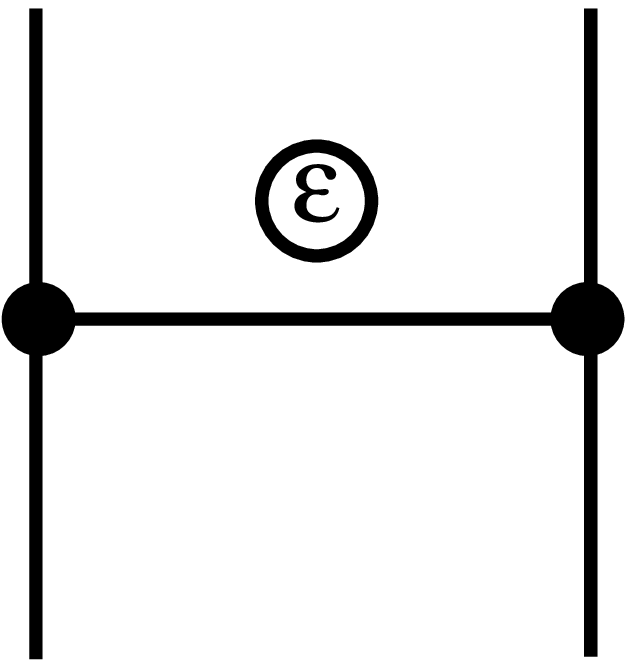}} \end{array}+\begin{array}{c}\scalebox{.15}{\psfig{figure=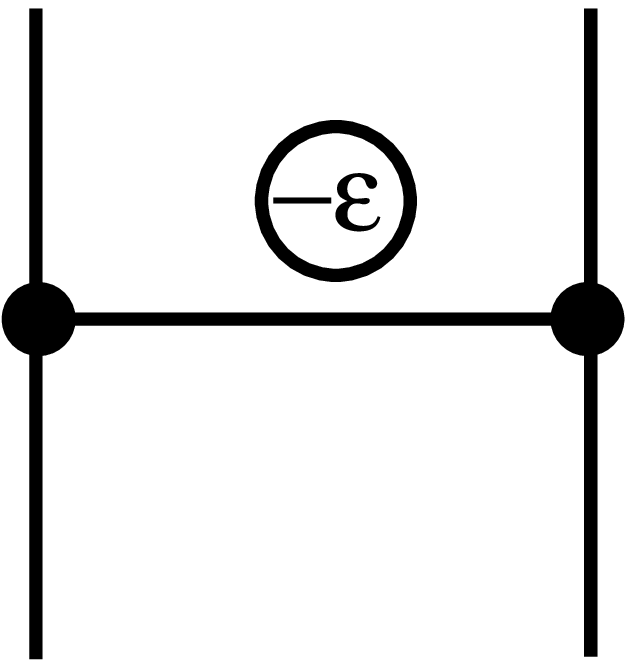}} \end{array}=0,
\]
\begin{eqnarray*}
\underline{\Delta \text{RIII}:}\,\,\,  \begin{array}{c}\scalebox{.15}{\psfig{figure=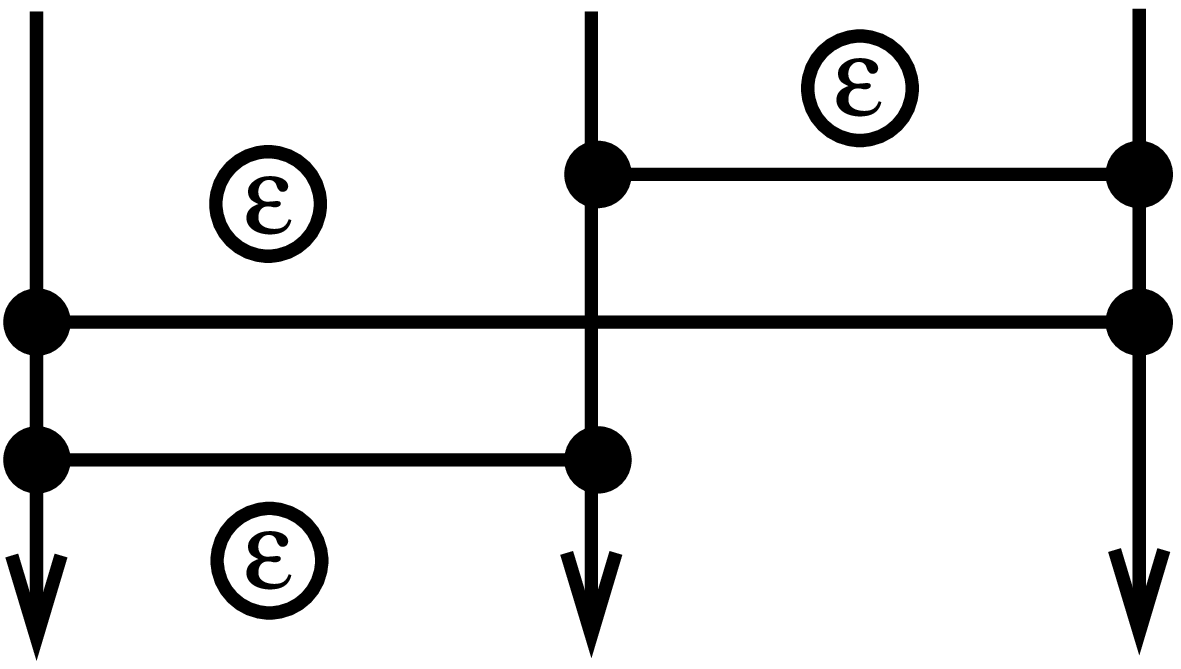}} \end{array}+\begin{array}{c}\scalebox{.15}{\psfig{figure=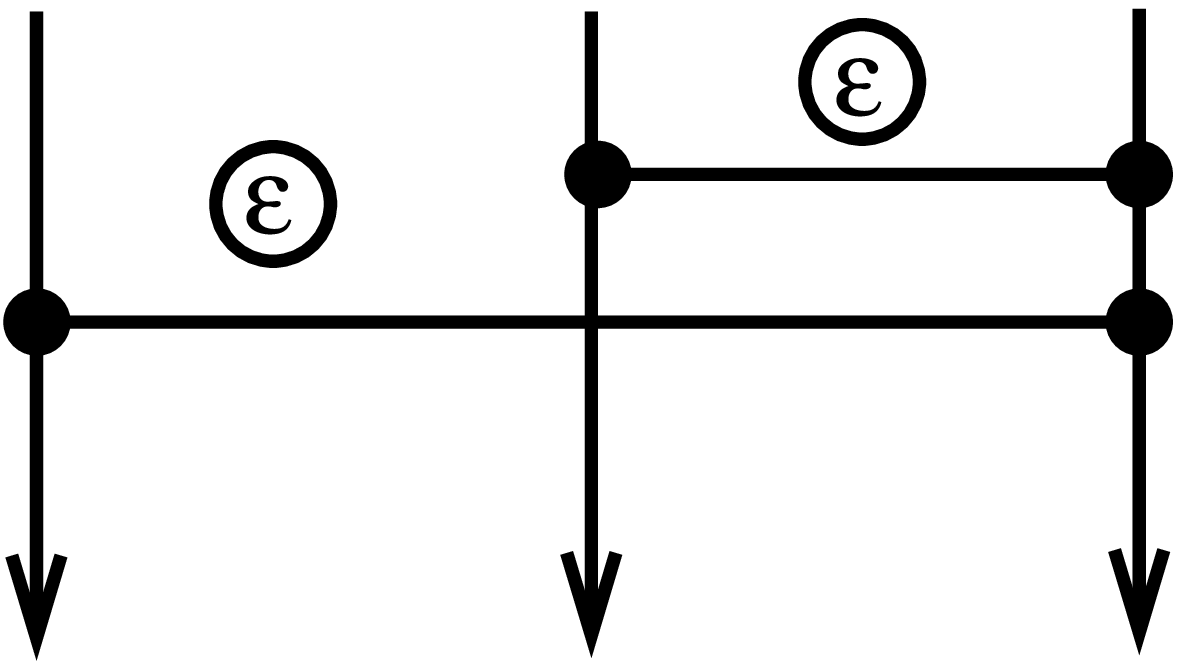}} \end{array}+\begin{array}{c}\scalebox{.15}{\psfig{figure=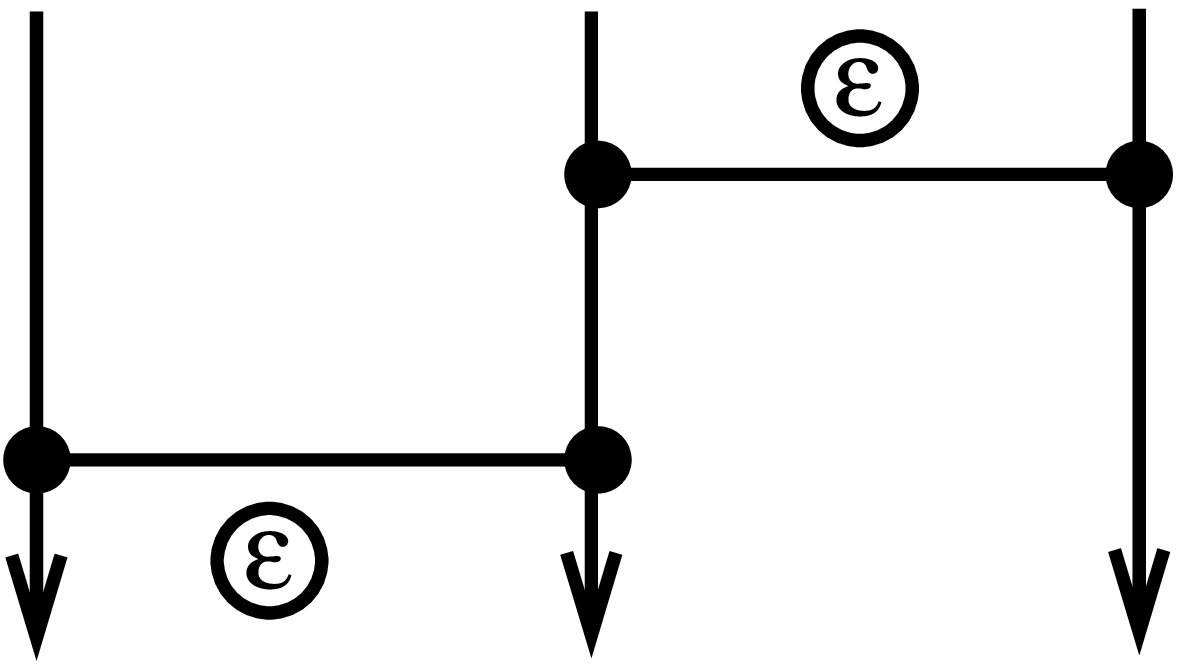}} \end{array}+\begin{array}{c}\scalebox{.15}{\psfig{figure=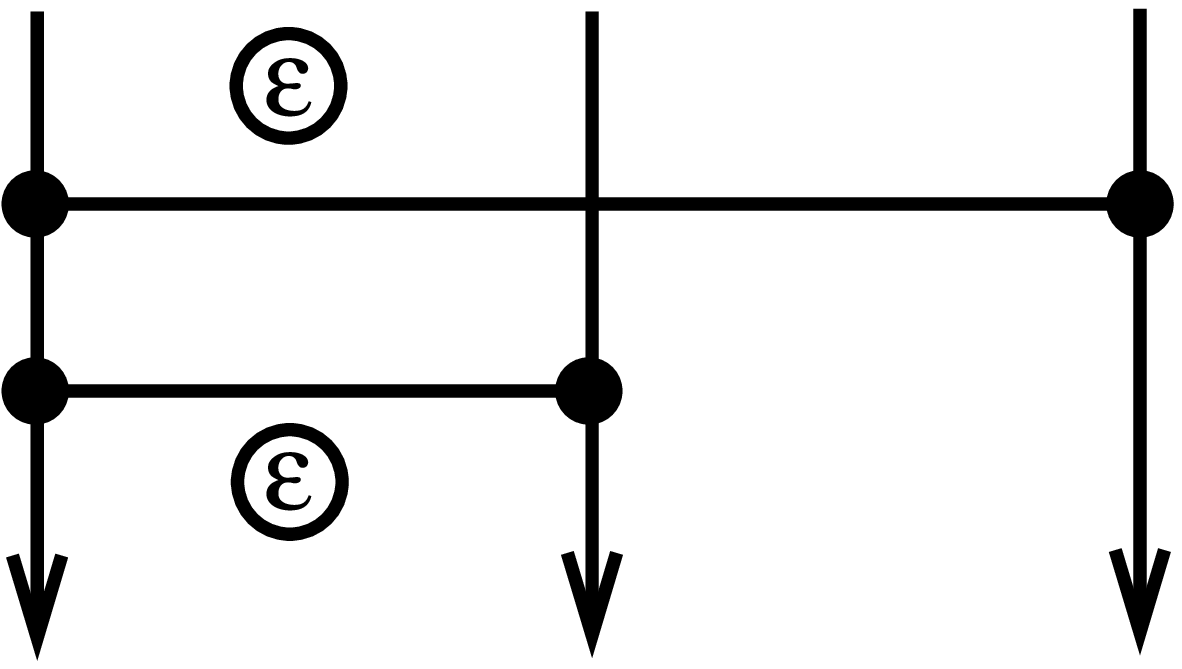}} \end{array} &=& \\ \begin{array}{c}\scalebox{.15}{\psfig{figure=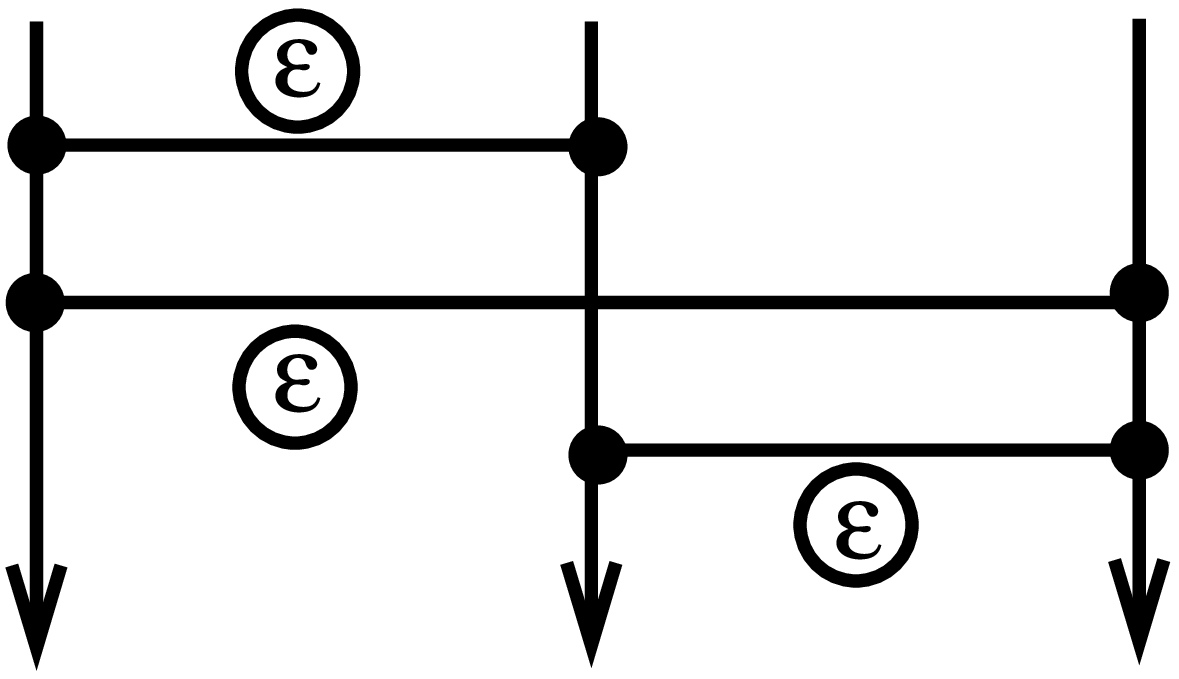}} \end{array}+\begin{array}{c}\scalebox{.15}{\psfig{figure=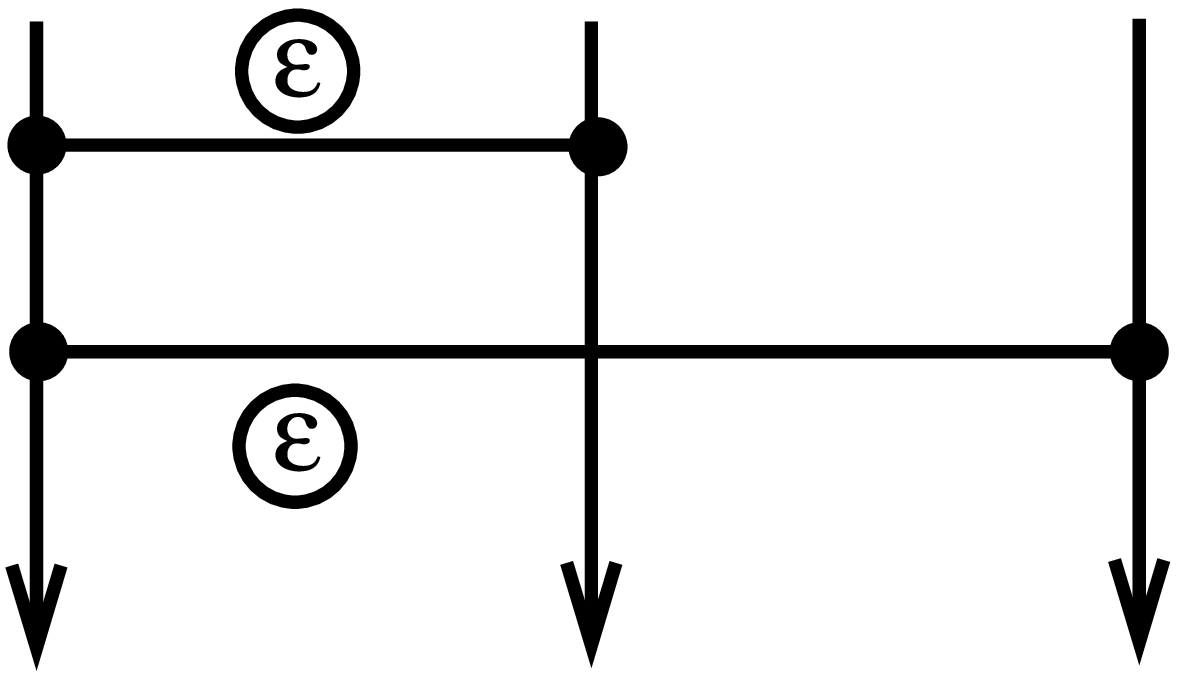}} \end{array}+\begin{array}{c}\scalebox{.15}{\psfig{figure=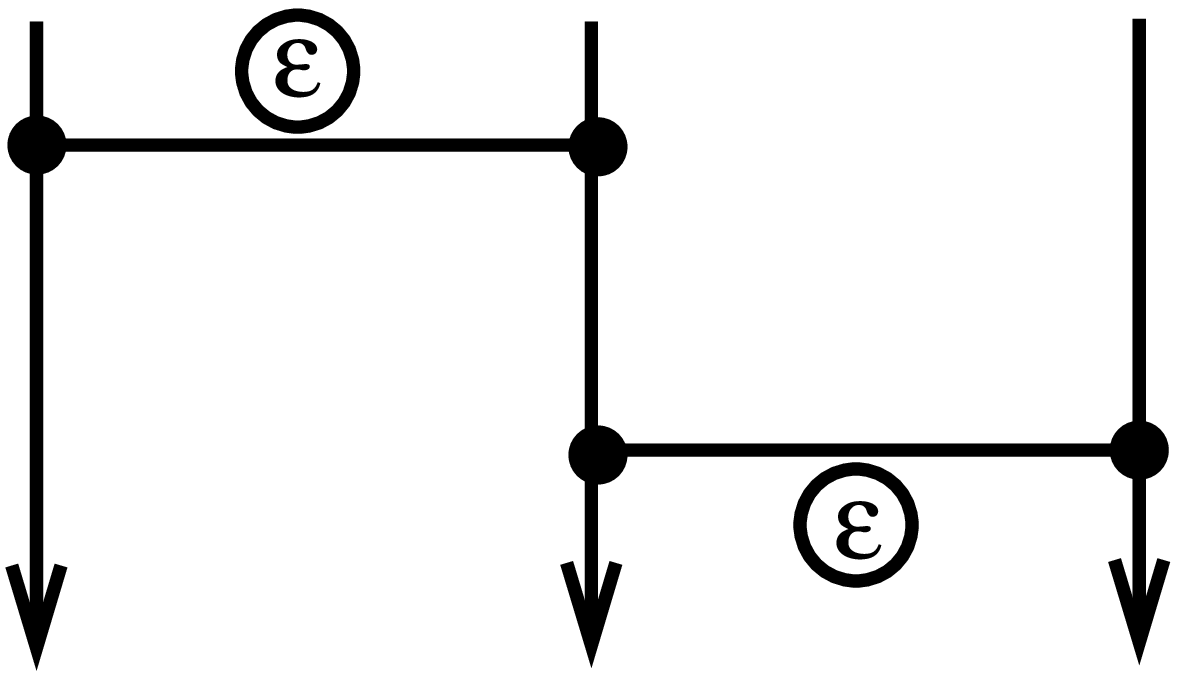}} \end{array}+\begin{array}{c}\scalebox{.15}{\psfig{figure=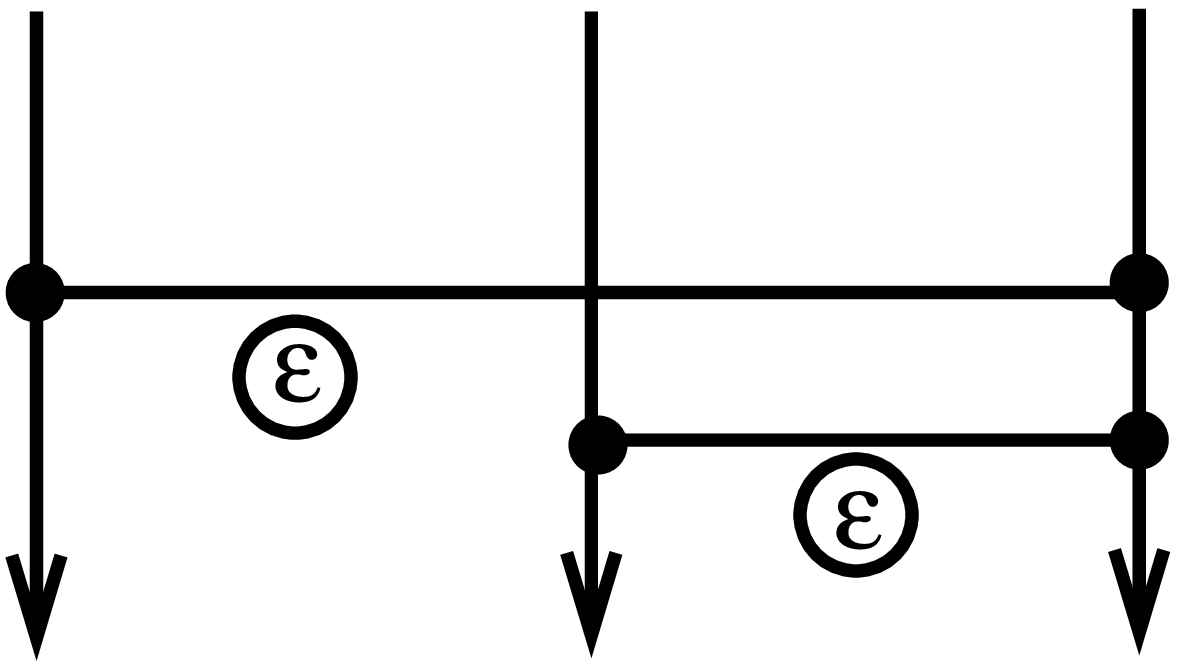}} \end{array} & & \\
\end{eqnarray*}
\caption{$\Delta \mathscr{R}$ relations} \label{deltarels}
\end{figure}

Let $\Delta \mathscr{R}=\left< \Delta \text{RI}, \Delta \text{RII}, \Delta \text{RIII} \right>$.  Define:
\[
\mathscr{V}=\frac{\mathbb{Z}[\mathscr{C}]}{\Delta \mathscr{R}},\,\,\, \mathscr{V}_n=\frac{\mathbb{Z}[\mathscr{C}]}{C_n+\Delta \mathscr{R}}
\]
Let $\mathscr{VK}$ denote the set of virtual knot \emph{diagrams} and $\mathscr{D}$ the set of Gauss diagrams.  Define $\hat{g}:\mathbb{Z}[\mathscr{VK}] \to \mathbb{Z}[\mathscr{D}]$ on generators to be the map which assigns to every virtual knot diagram its Gauss diagram.  Define $g:\mathbb{Z}[\mathscr{VK}]/\text{ker}(\hat{g}) \to \mathbb{Z}[\mathscr{D}]$ via Noether's First Isomorphism Theorem. For $D \in \mathscr{D}$, let $\bar{D} \in \mathscr{C}$ denote the chord diagram obtained from $D$ by erasing the arrow head of every arrow of $D$. We will also refer to this operation by the map $\text{Bar}:\mathbb{Z}[\mathscr{D}]\to\mathbb{Z}[\mathscr{C}]$.  Define $I:\mathbb{Z}[\mathscr{D}]\to \mathbb{Z}[\mathscr{C}]$ on generators $D \in \mathscr{D}$ by:
\[
I(D)=\sum_{D' \subset D} \overline{D'}=\overline{I_{\text{GPV}}(D)}
\]
where the sum is over all Gauss diagrams $D'$ obtained from $D$ be deleting a subset of its arrows.  
\begin{proposition} For all $v \in \text{Hom}_{\mathbb{Z}}(\mathscr{V},\mathbb{Q})$, $v \circ I$ is a virtual knot invariant that is invariant under the virtualization move.
\end{proposition}
\begin{proof} By the GPV theorem, it is sufficient to show that $v(\overline{\text{PI}})=v(\overline{\text{PII}})=v(\overline{\text{PIII}})=\{0\}$. However, this is clearly true since $\overline{\left< \text{PI}, \text{PII}, \text{PIII} \right>}=\Delta \mathscr{R}$.  For the second assertion, note that if $K$ and $K'$ are obtained from one another by a single virtualization move, then their Gauss diagrams differ only in the direction of a single arrow. In that case, $I(g(K))=I(g(K'))$. 
\end{proof}
\subsection{Universality of the model} In this section we establish the universality of $I$ and $I_n$.  The following lemma is useful in this regard.
\begin{lemma} \label{pok} Suppose that $v \in \text{Hom}_{\mathbb{Z}}(\mathscr{P},\mathbb{Q})$ is virtualization invariant.  In other words,
\[
v \circ I_{\text{GPV}} \left( \begin{array}{c} \scalebox{.25}{\psfig{figure=virtmove1.eps}} \end{array} \right)=v \circ I_{\text{GPV}} \left( \begin{array}{c} \scalebox{.25}{\psfig{figure=virtmove2.eps}} \end{array}\right)
\] 
Then for all signed dashed arrow diagrams $D$, if $D'$ is obtained from $D$ by changing the direction of one arrow, then $v(D)=v(D')$.
\end{lemma}
\begin{proof} The proof is by induction on the number of arrows of the dashed diagram $D$. If $n=1$, the result is obvious in the case of knots.  For long knots, we have:
\begin{eqnarray*}
v\circ I_{\text{GPV}}\left( \begin{array}{c} \scalebox{.15}{\psfig{figure=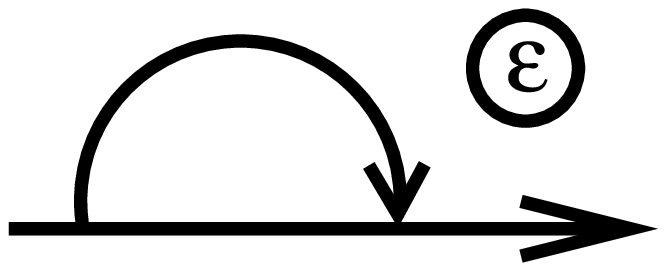}} \end{array} \right) &=& v \left( \begin{array}{c} \scalebox{.15}{\psfig{figure=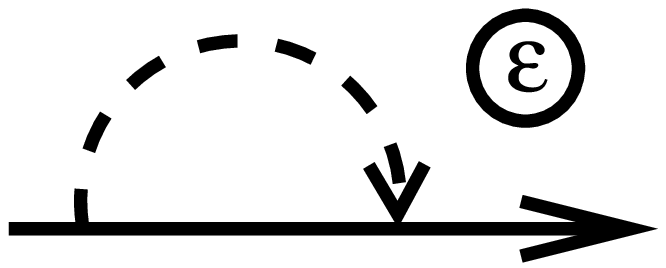}} \end{array}\right)+v\left( \begin{array}{c} \scalebox{.15}{\psfig{figure=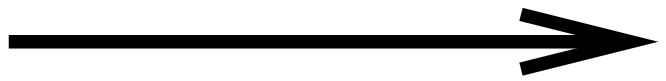}} \end{array}\right) \\
v\circ I_{\text{GPV}}\left( \begin{array}{c} \scalebox{.15}{\psfig{figure=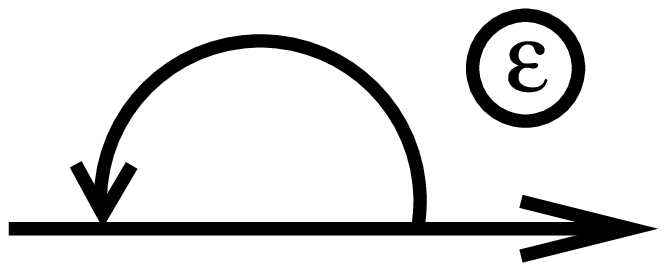}} \end{array}\right) &=& v \left(\begin{array}{c} \scalebox{.15}{\psfig{figure=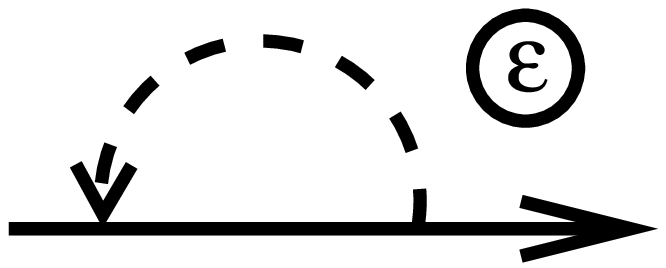}} \end{array} \right)+v\left(\begin{array}{c} \scalebox{.15}{\psfig{figure=unilemmNULL.eps}} \end{array} \right) \\
\Rightarrow v \left( \begin{array}{c} \scalebox{.15}{\psfig{figure=unilemmL2.eps}} \end{array}\right)&=&v\left(\begin{array}{c} \scalebox{.15}{\psfig{figure=unilemmR2.eps}} \end{array}\right) \\
\end{eqnarray*}
Suppose now that the theorem is true for $n$.  Let $D$ be a dashed arrow diagram with $n+1$ arrows, $D'$ a diagram obtained from $D$ by switching the direction of one arrow.  Let $B=i^{-1}(D)$ and $B'=i^{-1}(D')$.  Let $e$ denote the arrow whose direction is changed from $D$ to $D'$.
\begin{eqnarray*}
v\circ I_{\text{GPV}}(B) &=& v\left( \sum_{R \subset B} i(R) \right) \\
                         &=& v\left( \sum_{\stackrel{R\subset B}{e \in R}} i(R) \right)+v\left( \sum_{\stackrel{R\subset B}{e \notin R}} i(R)\right)\\
v\circ I_{\text{GPV}}(B') &=&  v\left( \sum_{\stackrel{R'\subset B'}{e \in R'}} i(R) \right)+v\left( \sum_{\stackrel{R'\subset B'}{e \notin R'}} i(R')\right)\\
\end{eqnarray*}
Now, for $e \in R$, let $R'$ denote the diagram obtained from $R$ by switching the direction of $e$.  By applying linearity of $v$ and subtracting the two equations of interest, we obtain:
\[
\sum_{\stackrel{R\subset B}{e \in R}} v(i(R))-v(i(R'))=0
\]
Now, for $R \subset B$, $R$ has between $1$ and $n+1$ arrows. Since $R$ and $R'$ differ only in the direction of a single arrow the induction hypothesis implies that $v(i(R))=v(i(R'))$ for all $R$ having $\le n$ arrows.  Thus,
\[
0=v(i(B))-v(i(B'))=v(D)-v(D')
\]
This establishes the lemma.
\end{proof}
\begin{theorem}\label{univ} The map $I:\mathbb{Z}[\mathscr{D}] \to \mathscr{V}$ is universal in the sense that if $v \in \text{Hom}_{\mathbb{Z}}(\mathscr{P},\mathbb{Q})$ is virtualization invariant, then there is a $v' \in \text{Hom}_{\mathbb{Z}}(\mathscr{V},\mathbb{Q})$ such that the following diagram commutes:
\[
\xymatrix{ \mathbb{Z}[\mathscr{D}] \ar[r]^I \ar[d]_{I_{\text{GPV}}} & \mathscr{V} \ar@{-->}[d]^{v'} \\
\mathscr{P} \ar[r]_v & \mathbb{Q}}              
\] 
\end{theorem}
\begin{proof} Let $C \in \mathscr{C}$ and let $\vec{C}$ denote the signed arrow diagram obtained from $C$ by directing the chords of $C$.  Define $v'(C)=v(\vec{C})$.  Note that by Lemma \ref{pok}, if $\vec{C'}$ is an arrow diagram with $\text{Bar}(\vec{C'})=C$, then $v(\vec{C})=v(\vec{C'})$. Thus, $v'$ is well-defined on $\mathbb{Z}[\mathscr{C}]$.

To complete the proof, it is only necessary to show that $v'(r)=0$ for all $r\in \Delta \mathscr{R}$. For each $r\in \Delta \text{RI},\Delta \text{RII}$, or $\Delta \text{RIII}$, there is a $\vec{r} \in \text{Bar}^{-1}(r)$ such that $\vec{r} \in \text{PI},\text{PII}$, or $\text{PIII}$, respectively.  For example, we have:
\begin{eqnarray*}
r &=&  \begin{array}{c}\scalebox{.15}{\psfig{figure=chordR2_3.eps}} \end{array}+\begin{array}{c}\scalebox{.15}{\psfig{figure=chordR2_1.eps}} \end{array}+\begin{array}{c}\scalebox{.15}{\psfig{figure=chordR2_2.eps}} \end{array}\\
\vec{r} &=& \begin{array}{c}\scalebox{.15}{\psfig{figure=polyak2_1.eps}} \end{array}+\begin{array}{c}\scalebox{.15}{\psfig{figure=polyak2_2.eps}} \end{array}+\begin{array}{c}\scalebox{.15}{\psfig{figure=polyak2_3.eps}} \end{array} \\
\end{eqnarray*}
Since $v(\text{PI})=v(\text{PII})=v(\text{PIII})=0$, it follows that $v'(r)=0$.
\end{proof}

\begin{lemma} \label{pnok} For $v \in \text{Hom}_{\mathbb{Z}}(\mathscr{V}_n,\mathbb{Q})$, $v\circ I$ is a GPV finite-type invariant of order $\le n$.
\end{lemma}
\begin{proof} For $D \in \mathscr{P}$, $v \circ I\circ I_{\text{GPV}}^{-1}(D)=v(\overline{D})$.  If $D$ has more than $n$ dashed arrows, then $v(\overline{D})=0$.
\end{proof}
Let $p_n:\mathscr{V} \to \mathscr{V}_n$ denote that natural projection.  Define $I_n:\mathbb{Z}[D] \to \mathscr{V}_n$ to be the composition: $I_n=p_n \circ I$.

\begin{theorem} The map $I_n:\mathbb{Z}[\mathscr{D}] \to \mathscr{V}_n$ is universal in the sense that for all $v \in \text{Hom}(\mathscr{P}_n,\mathbb{Q})$ such that $v \circ (I_{\text{GPV}})_n$ is virtualization invariant, then there is a $v'\in \text{Hom}_{\mathbb{Z}}(\mathscr{V}_n,\mathbb{Q})$ such that the following diagram commutes.
\[
\xymatrix{ \mathbb{Z}[\mathscr{D}] \ar[r]^{I_n} \ar[d]_{(I_{\text{GPV}})_n} & \mathscr{V}_n \ar@{-->}[d]^{v'} \\
\mathscr{P}_n \ar[r]_v & \mathbb{Q}}              
\] 
\end{theorem}
\begin{proof} This follows immediately from Lemma \ref{pnok} and Theorem \ref{univ}.
\end{proof}
It is obvious that constant invariants are of GPV finite type of every order.  For virtual knots, they are generated by the combinatorial formula $\left<\bigcirc,\cdot\right>$.  For long virtual knots, they are generated by the combinatorial formula $<\underline{\hspace{.25cm}}\,\,,\cdot>$.  Denote both of these generators by $1$.
\begin{lemma} \label{lowo} For small orders, the following computations hold:
\begin{enumerate}
\item For virtual knots, $\text{Hom}_{\mathbb{Z}}(\mathscr{V}_3,\mathbb{Q})=<1>$
\item For long virtual knots, $\text{Hom}_{\mathbb{Z}}(\mathscr{V}_2,\mathbb{Q})=<1>$
\end{enumerate} 
\end{lemma}
\begin{proof} For the first assertion, we have from \cite{MR1763963} that 1 and the following formula generate the GPV finite-type invariants of order $\le 3$. 
\[
\left<3 \cdot \begin{array}{c}\scalebox{.15}{\psfig{figure=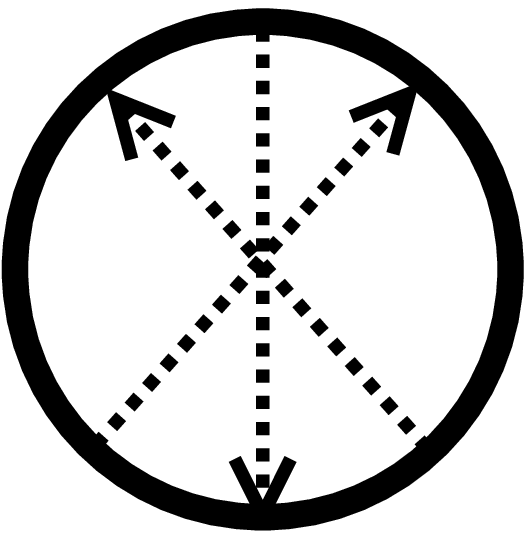}} \end{array}-\begin{array}{c}\scalebox{.15}{\psfig{figure=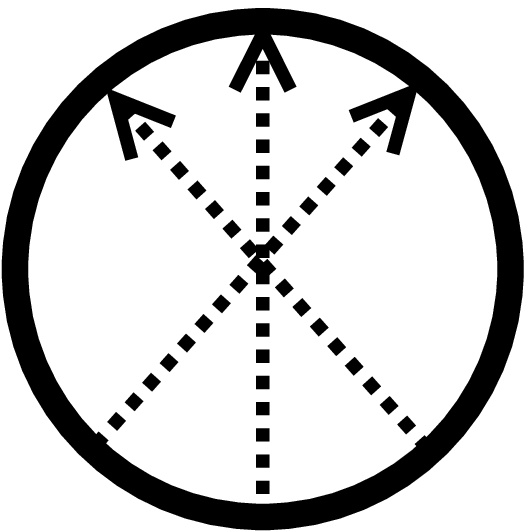}} \end{array}+\begin{array}{c}\scalebox{.15}{\psfig{figure=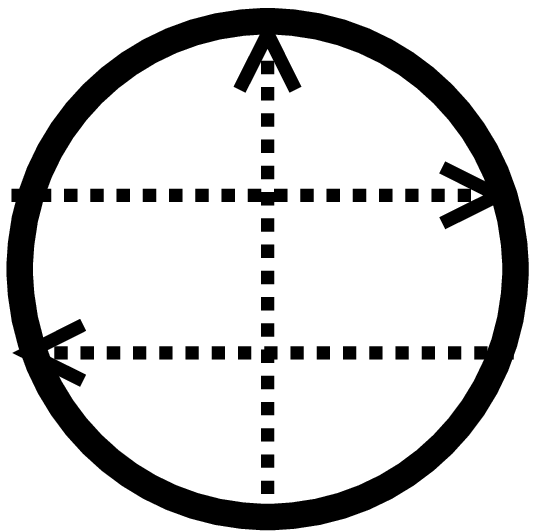}} \end{array}+\begin{array}{c}\scalebox{.15}{\psfig{figure=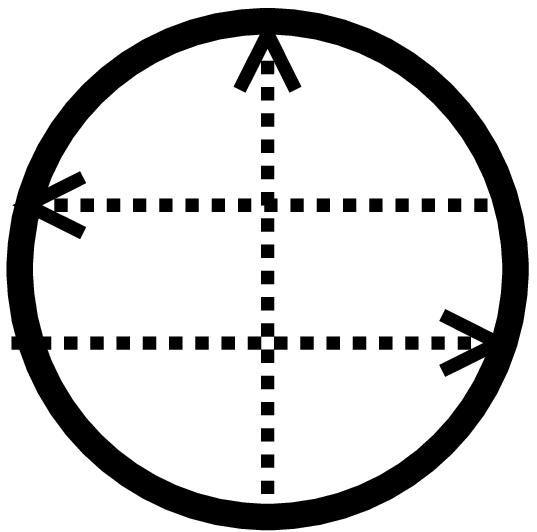}} \end{array}-\begin{array}{c}\scalebox{.15}{\psfig{figure=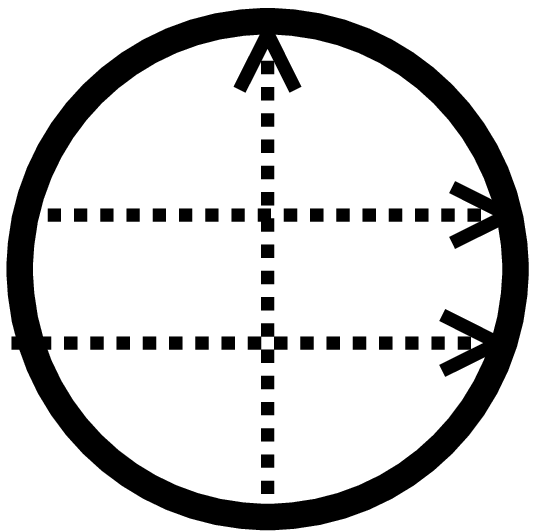}} \end{array}-\begin{array}{c}\scalebox{.15}{\psfig{figure=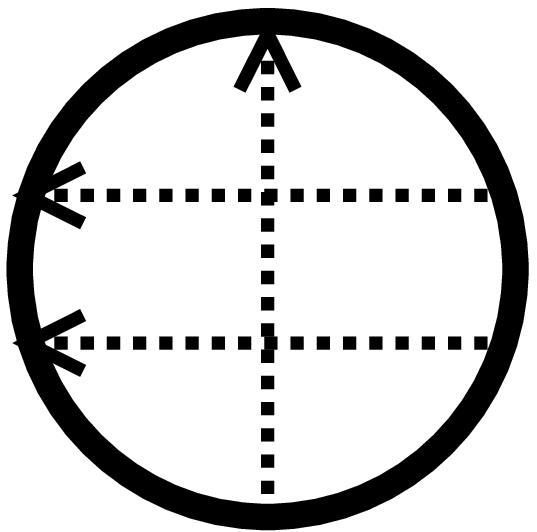}} \end{array}-\begin{array}{c}\scalebox{.15}{\psfig{figure=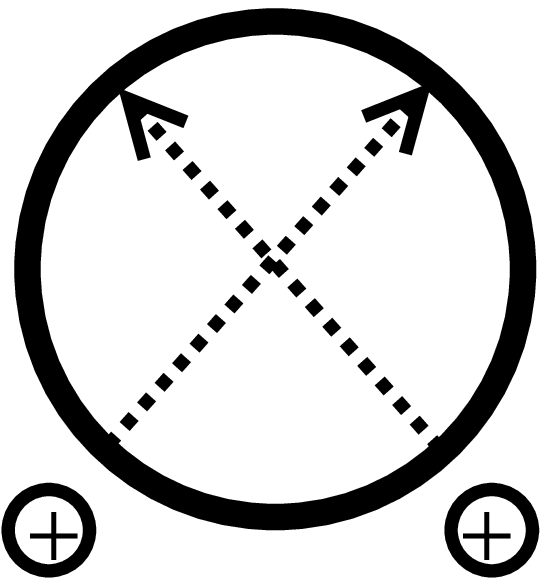}} \end{array}+\begin{array}{c}\scalebox{.15}{\psfig{figure=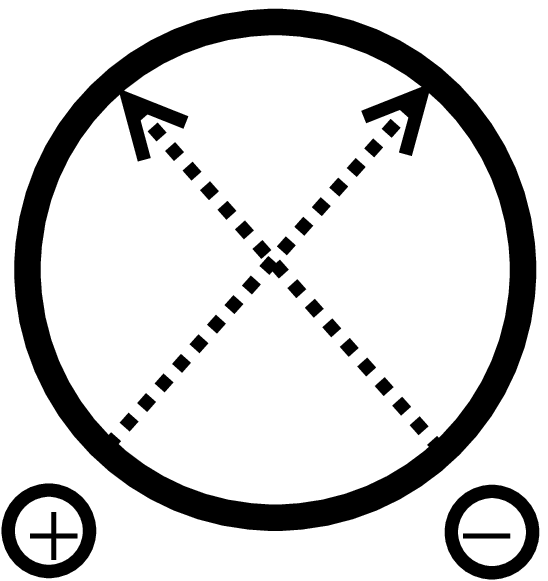}} \end{array}, \,\,\,\, \cdot \right>
\]
Since the value on diagrams 3 and 5 is different, Lemma \ref{pok} implies our result.

For the second assertion, we have from \cite{MR1763963} that 1 and the following two formulas generate the GPV finite-type invariants of order $\le 2$.
\[
\left< \begin{array}{c}\scalebox{.25}{\psfig{figure=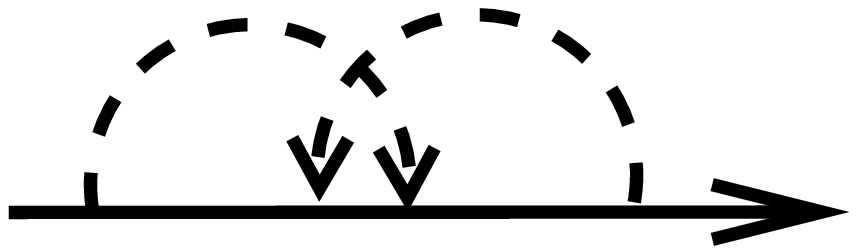}} \end{array},\,\,\, \cdot \right>,\,\,\,\left< \begin{array}{c}\scalebox{.25}{\psfig{figure=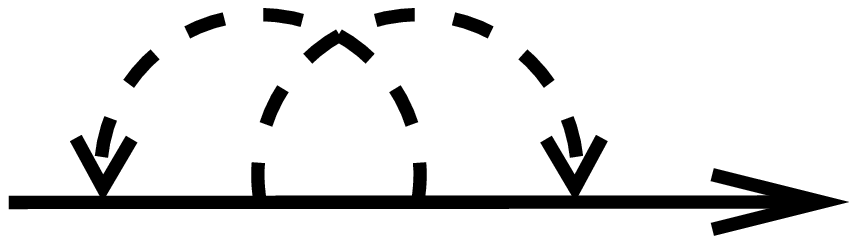}} \end{array},\,\,\, \cdot \right>
\]
In this case, Lemma \ref{pok} also implies our result.
\end{proof}
\subsection{Algebraic decomposition of $\text{Hom}_{\mathbb{Z}}(\mathscr{P}_n,\mathbb{Q})$ and $\text{Hom}_{\mathbb{Z}}(\mathscr{V}_n,\mathbb{Q})$} The hero of the decomposition is Polyak's algebra of arrow diagrams. Its precise relationship to the groups $\mathscr{V}_n$ and $\mathscr{P}_n$ is what allows us to obtain Theorem \ref{bigthm}. While this relationship is already well-known (see \cite{Polyak}), it is prudent to describe it carefully here.   

Define $\vec{\mathscr{F}}_n^{\pm}$, $\vec{\mathscr{F}}_n$ to be the free abelian group generated by the set of signed arrow diagrams and unsigned arrow diagrams having exactly $n$ arrows, respectively. Define $\overline{\mathscr{F}}_n^{\pm}$, $\overline{\mathscr{F}}_n$ to be the free abelian group generated by the set of signed chord diagrams and unsigned chord diagrams having exactly $n$ chords, respectively.  

The decomposition is the same for the $\mathscr{V}_n$ and $\mathscr{P}_n$.  Therefore we define variables which stand in place of either case.
\newline
\centerline{
\begin{tabular}{c|c|c}
 & Chord Case & Arrow Case \\ \hline
$\mathscr{B}_n$ & $\mathscr{V}_n$ & $\mathscr{P}_n$ \\
$B_n$ & $C_n$ & $A_n$ \\
$R$ & $\Delta \mathscr{R}$ & $\Delta \mathscr{P}$ \\
$\mathscr{F}_n^{\pm}$ & $\overline{\mathscr{F}}_n^{\pm}$ & $\vec{\mathscr{F}}_n^{\pm}$ \\
$\mathscr{F}_n$ & $\overline{\mathscr{F}}_n$ & $\vec{\mathscr{F}}_n$ \\
\end{tabular}
}
In this section, we investigate the following short exact sequence:
\[
\xymatrix{0 \ar[r] & \frac{B_n+R}{B_{n+1}+R} \ar[r] & \mathscr{B}_{n+1} \ar[r]^{\pi_n} & \mathscr{B}_n \ar[r] & 0}
\]
and (more importantly), its dual:
\[
\xymatrix{0 \ar[r] & \text{Hom}_{\mathbb{Z}}\left(\mathscr{B}_n,\mathbb{Q}\right) \ar[r]^{\pi_n^*} & \text{Hom}_{\mathbb{Z}}\left(\mathscr{B}_{n+1},\mathbb{Q}\right) \ar[r] & \text{Hom}_{\mathbb{Z}}\left(\frac{B_n+R}{B_{n+1}+R},\mathbb{Q}\right)}
\]
In the next section, we will show that the rightmost module in the dual sequence vanishes and hence $\pi_n^*$ is an isomorphism. For this, we use some intermediate groups which are isomorphic to the groups in the grading of Polyak's algebra of arrow diagrams.  The relations for the intermediate groups are given below.
\[
\underline{\vec{\text{1T}}_n:} \begin{array}{c} \scalebox{.1}{\psfig{figure=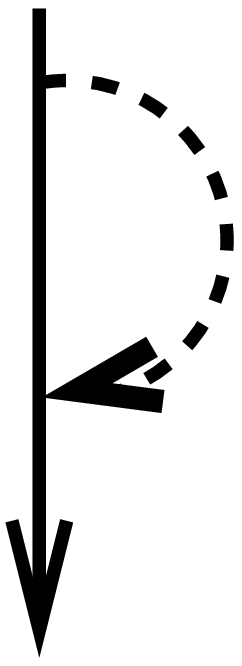}} \end{array}=0 ,\,\, \underline{\vec{\text{NS}}_n:} \begin{array}{c} \scalebox{.1}{\psfig{figure=polyak2_2.eps}} \end{array}+\begin{array}{c} \scalebox{.1}{\psfig{figure=polyak2_3.eps}} \end{array}=0
\]
\begin{eqnarray*}
\underline{\vec{\text{6T}}_n^{\pm}:}  \begin{array}{c}\scalebox{.1}{\psfig{figure=polyak3_2.eps}} \end{array}+\begin{array}{c}\scalebox{.1}{\psfig{figure=polyak3_3.eps}} \end{array}+\begin{array}{c}\scalebox{.1}{\psfig{figure=polyak3_4.eps}} \end{array} &=& \begin{array}{c}\scalebox{.1}{\psfig{figure=polyak3_6.eps}} \end{array}+\begin{array}{c}\scalebox{.1}{\psfig{figure=polyak3_7.eps}} \end{array}+\begin{array}{c}\scalebox{.1}{\psfig{figure=polyak3_8.eps}} \end{array}  \\
\end{eqnarray*}
\begin{eqnarray*}
\underline{\vec{\text{6T}}_n:}  \begin{array}{c}\scalebox{.1}{\psfig{figure=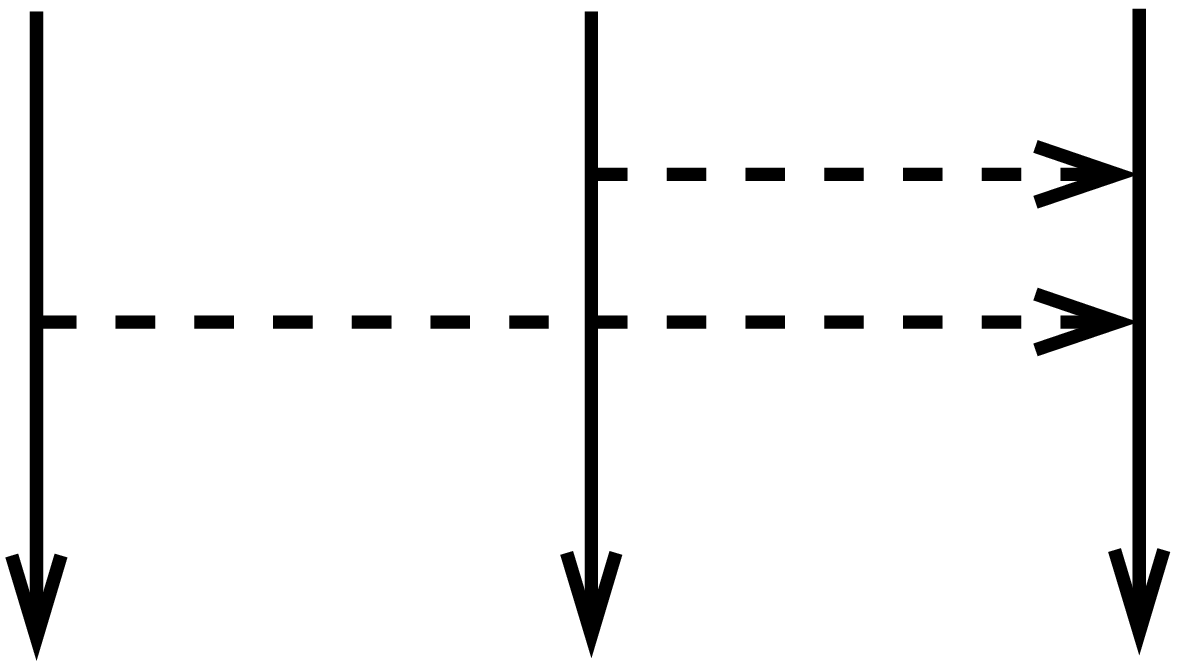}} \end{array}+\begin{array}{c}\scalebox{.1}{\psfig{figure=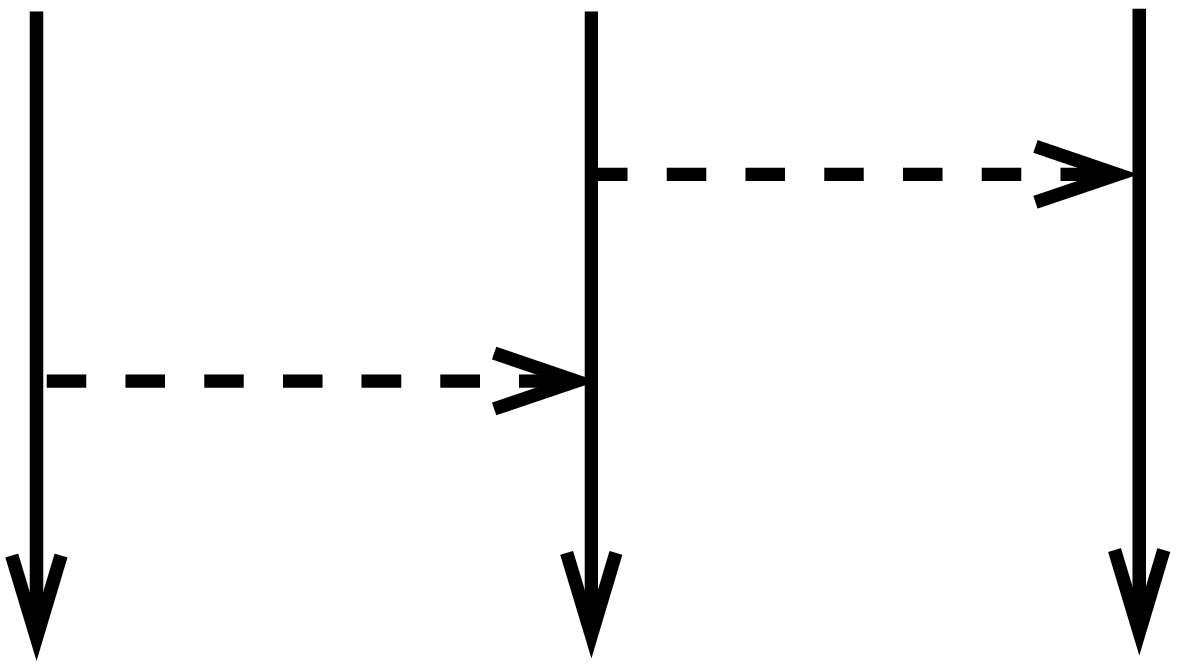}} \end{array}+\begin{array}{c}\scalebox{.1}{\psfig{figure=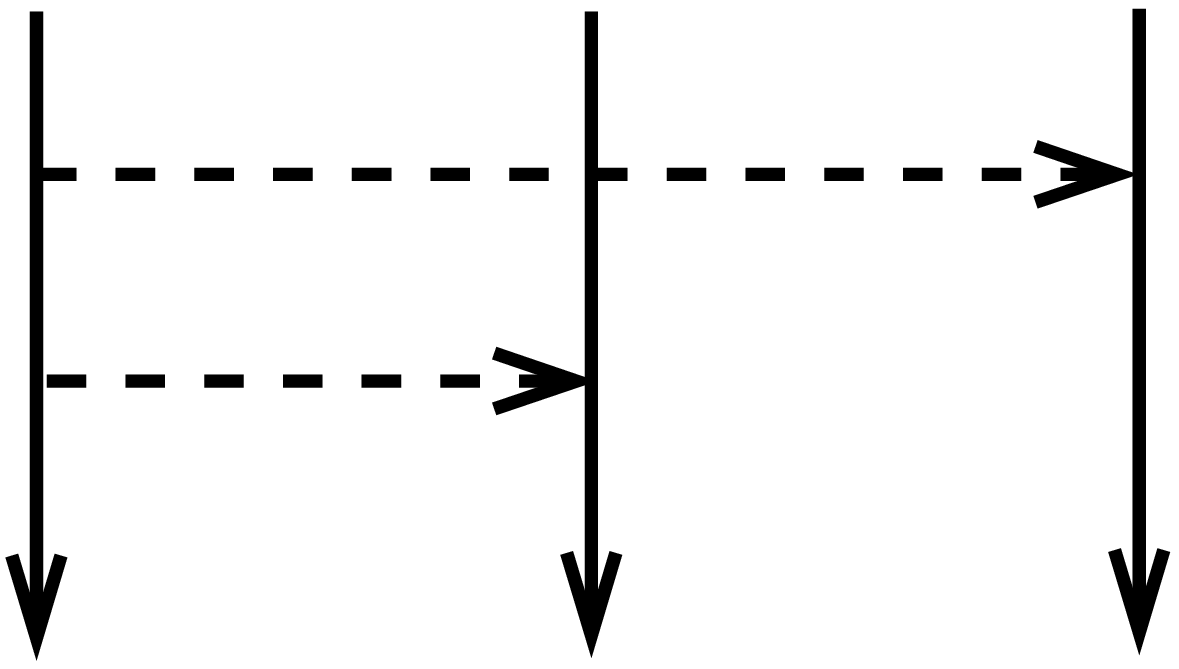}} \end{array} &=&  \begin{array}{c}\scalebox{.1}{\psfig{figure=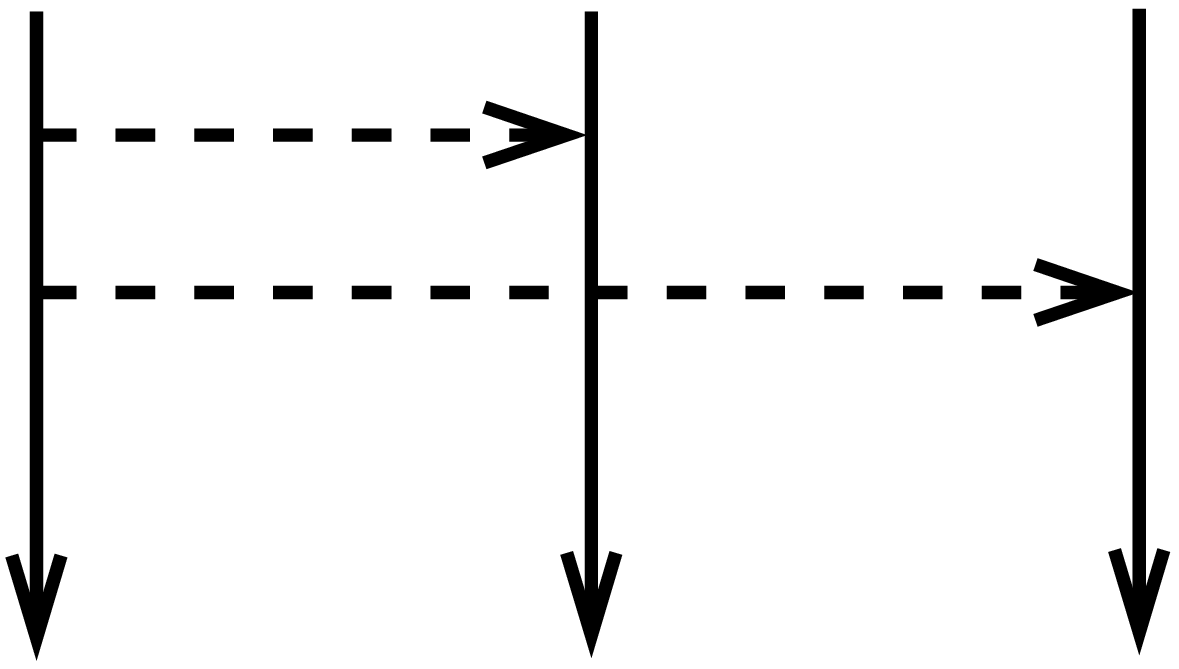}} \end{array}+\begin{array}{c}\scalebox{.1}{\psfig{figure=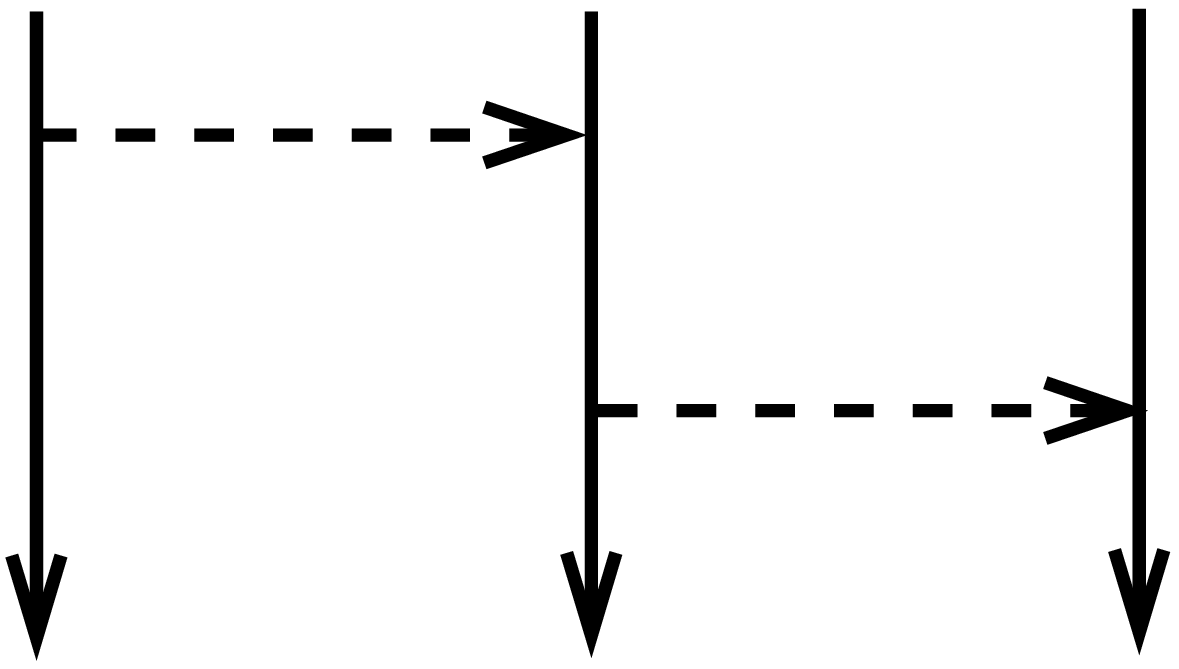}} \end{array}+\begin{array}{c}\scalebox{.1}{\psfig{figure=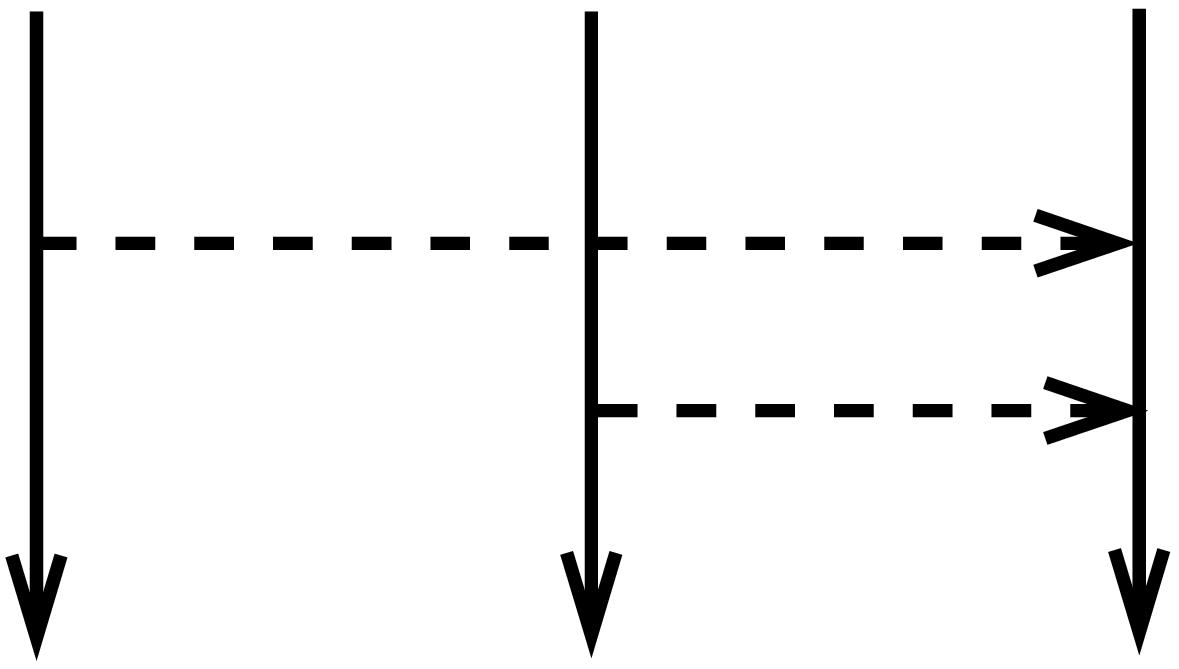}} \end{array} \\
\end{eqnarray*}
Using the same convention as above, define $\overline{\text{1T}}_n^{\pm}=\text{Bar}(\Delta\text{PI}_n)$,      \,$\overline{\text{NS}}_n=\text{Bar}(\vec{\text{NS}}_n)$,\,    $\overline{\text{6T}}_n^{\pm}=\text{Bar}(\vec{\text{6T}}_n^{\pm})$, and $\overline{\text{6T}}_n=\text{Bar}(\vec{\text{6T}}_n)$. Define $\text{6T}_n=\overline{\text{6T}}_n$ or $\vec{\text{6T}}_n$, $\text{6T}_n^{\pm}=\overline{\text{6T}}_n^{\pm}$ or $\vec{\text{6T}}_n^{\pm}$, $\text{1T}_n^{\pm}=\overline{\text{1T}}_n^{\pm}$ or $\Delta\text{PI}_n$, and $\text{NS}_n=\overline{\text{NS}}_n$ or $\vec{\text{NS}}_n$.
\begin{theorem} There is an injection:
\[
\text{Hom}_{\mathbb{Z}}\left( \frac{B_{n-1}+R}{B_n+R}, \mathbb{Q} \right)\hookrightarrow \text{Hom}_{\mathbb{Z}} \left(\frac{\mathscr{F}_n}{\left<\text{1T}_n,\text{6T}_n\right>} ,\mathbb{Q} \right)
\]
\end{theorem}
\begin{proof} First we rewrite $\frac{B_{n-1}+R}{B_n+R}$ using Noether's Second Isomorphism Theorem.
\begin{eqnarray*}
\frac{B_{n-1}+R}{B_n+R} &=& \frac{\mathscr{F}_n^{\pm}+(B_n+R)}{B_n+R} \\
          &\cong& \frac{\mathscr{F}_n^{\pm}}{\mathscr{F}_n^{\pm} \cap (B_n+R)} \\
\end{eqnarray*}
\begin{lemma} $\left< \text{6T}_n^{\pm}, \text{1T}_n^{\pm}, \text{NS}_n \right> \subset \mathscr{F}_n^{\pm} \cap (B_n+R)$
\end{lemma}
\begin{proof} It is certainly true that $\text{LHS} \subset \mathscr{F}_n^{\pm}$. We will show that $\text{LHS} \subset B_n+R$ for the chord diagram case only:
\begin{eqnarray*}
\underline{\overline{\text{6T}}_n^{\pm}:}\,\,\, \begin{array}{c}\scalebox{.1}{\psfig{figure=chordR3_2.eps}} \end{array} &+&\begin{array}{c}\scalebox{.1}{\psfig{figure=chordR3_3.eps}} \end{array}+\begin{array}{c}\scalebox{.1}{\psfig{figure=chordR3_4.eps}} \end{array} - \begin{array}{c}\scalebox{.1}{\psfig{figure=chordR3_6.eps}} \end{array}-\begin{array}{c}\scalebox{.1}{\psfig{figure=chordR3_7.eps}} \end{array}-\begin{array}{c}\scalebox{.1}{\psfig{figure=chordR3_8.eps}} \end{array} \\
&=& \begin{array}{c}\scalebox{.1}{\psfig{figure=chordR3_1.eps}} \end{array}+\begin{array}{c}\scalebox{.1}{\psfig{figure=chordR3_2.eps}} \end{array}+\begin{array}{c}\scalebox{.1}{\psfig{figure=chordR3_3.eps}} \end{array}+\begin{array}{c}\scalebox{.1}{\psfig{figure=chordR3_4.eps}} \end{array}\\
&-&\left.\begin{array}{c}\scalebox{.1}{\psfig{figure=chordR3_5.eps}} \end{array}-\begin{array}{c}\scalebox{.1}{\psfig{figure=chordR3_6.eps}} \end{array}-\begin{array}{c}\scalebox{.1}{\psfig{figure=chordR3_7.eps}} \end{array}-\begin{array}{c}\scalebox{.1}{\psfig{figure=chordR3_8.eps}} \end{array} \right\} \in \Delta \mathscr{R}\\
&+& \left.\left(- \begin{array}{c}\scalebox{.1}{\psfig{figure=chordR3_1.eps}} \end{array}+\begin{array}{c}\scalebox{.1}{\psfig{figure=chordR3_5.eps}} \end{array} \right) \right\} \in C_n
\end{eqnarray*}
\[
\underline{\overline{\text{NS}}_n:}\,\,\, \begin{array}{c} \scalebox{.1}{\psfig{figure=chordR2_1.eps}} \end{array}+\begin{array}{c} \scalebox{.1}{\psfig{figure=chordR2_2.eps}} \end{array} =
\underbrace{\begin{array}{c}\scalebox{.1}{\psfig{figure=chordR2_3.eps}} \end{array}+\begin{array}{c}\scalebox{.1}{\psfig{figure=chordR2_1.eps}} \end{array}+\begin{array}{c}\scalebox{.1}{\psfig{figure=chordR2_2.eps}} \end{array}}_{ \in \Delta \mathscr{R}}-\underbrace{\begin{array}{c}\scalebox{.1}{\psfig{figure=chordR2_3.eps}} \end{array}}_{\in C_n}
\]
This completes the proof of the lemma.
\end{proof}
Therefore there is an exact sequence:
\[
\xymatrix{
\frac{\mathscr{F}_n^{\pm}}{\left< \text{6T}_n^{\pm}, \text{1T}_n^{\pm}, \text{NS}_n \right>} \ar[r] & \frac{\mathscr{F}_n^{\pm}}{\mathscr{F}_n^{\pm} \cap (B_n+R)} \ar[r] & 0}
\]
Define $\Xi:\mathscr{F}_n^{\pm} \to \mathscr{F}_n$ as in \cite{Polyak}.  For $F \in \mathscr{F}_n^{\pm}$, let $m(F)$ denote the number of $\ominus$ signs appearing in the diagram.  Denote by $|F|\in\mathscr{F}_n$ the diagram obtained by deleting all the signs of $F$.  Then $\Xi$ is defined on generators by:
\[
\Xi(F)=(-1)^{m(F)}|F|
\]
Extend $\Xi$ to all of $\mathscr{F}_n^{\pm}$ using linearity. It is clear that:
\begin{eqnarray*}
\Xi(\left<\text{6T}_n^{\pm}\right>) &\subset& \left< \text{6T}_n \right>\\
\Xi(\left<\text{1T}_n^{\pm}\right>) &\subset& \left< \text{1T}_n \right> \\
\Xi(\left<\text{NS}_n \right>) &=& \{0\}  \\
\end{eqnarray*}
The surjection $\mathscr{F}_n^{\pm} \to \mathscr{F}_n \to \mathscr{F}_n/\left<\text{1T}_n,\text{6T}_n \right>$ has kernel $\left< \text{6T}_n^{\pm}, \text{1T}_n^{\pm}, \text{NS}_n \right>$. Taking the dual of the above short exact sequence gives the injection:
\[
\text{Hom}_{\mathbb{Z}}\left( \frac{B_{n-1}+R}{B_n+R}, \mathbb{Q} \right)\hookrightarrow \text{Hom}_{\mathbb{Z}}\left(\frac{\mathscr{F}_n^{\pm}}{\left< \text{6T}_n^{\pm}, \text{1T}_n^{\pm}, \text{NS}_n \right>} ,\mathbb{Q}\right) \cong \text{Hom}_{\mathbb{Z}} \left(\frac{\mathscr{F}_n}{\left<\text{1T}_n,\text{6T}_n\right>} ,\mathbb{Q} \right)
\]
\end{proof}
As we shall shortly see, it is convenient to change from $\mathbb{Z}$-modules to vector spaces over $\mathbb{Q}$.  The following computation shows that this does not affect our result.
\begin{eqnarray*}
\text{Hom}_{\mathbb{Z}} \left(\frac{\mathscr{F}_n}{\left<\text{1T}_n,\text{6T}_n\right>} ,\mathbb{Q} \right) &\cong& \text{Hom}_{\mathbb{Z}} \left(\frac{\mathscr{F}_n}{\left<\text{1T}_n,\text{6T}_n\right>} ,\text{Hom}_{\mathbb{Z}}(\mathbb{Q},\mathbb{Q}) \right)\\
&\cong& \text{Hom}_{\mathbb{Z}} \left(\frac{\mathscr{F}_n}{\left<\text{1T}_n,\text{6T}_n\right>}\otimes \mathbb{Q} ,\mathbb{Q} \right)
\end{eqnarray*}
\subsection{Proof of Theorem 1} To prove the theorem, we must also consider the four term relation in $\overline{\mathscr{F}}_n$:
\[
\underline{\overline{\text{4T}}_n:} \,\,\, \begin{array}{c}\scalebox{.15}{\psfig{figure=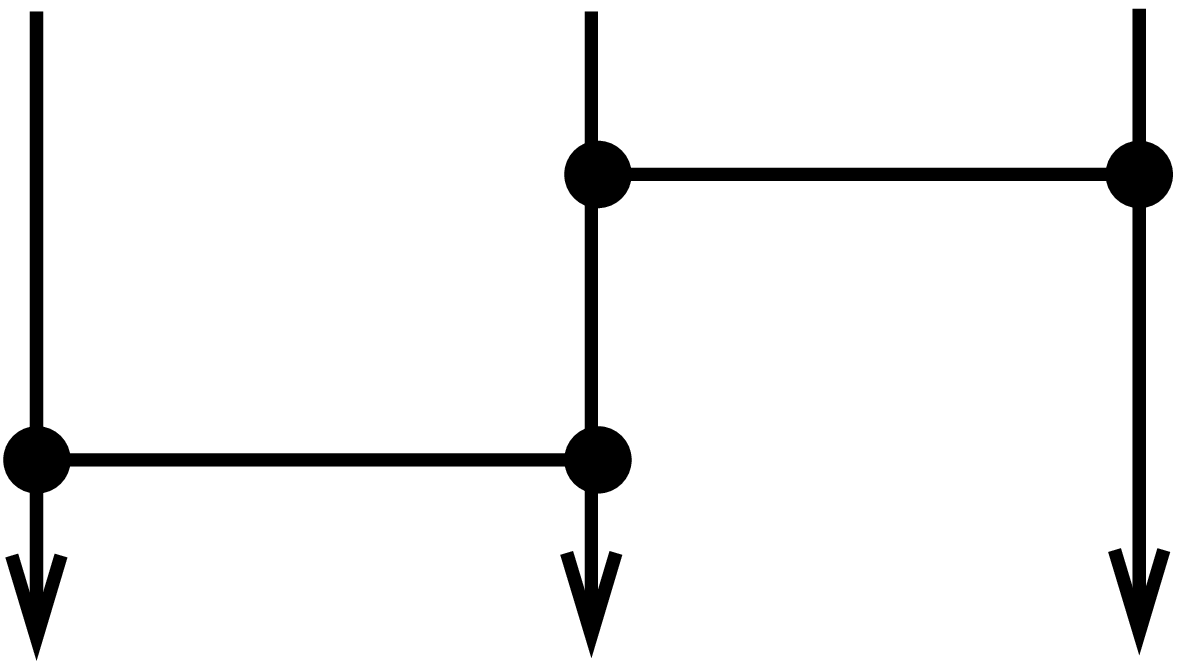}} \end{array}-\begin{array}{c}\scalebox{.15}{\psfig{figure=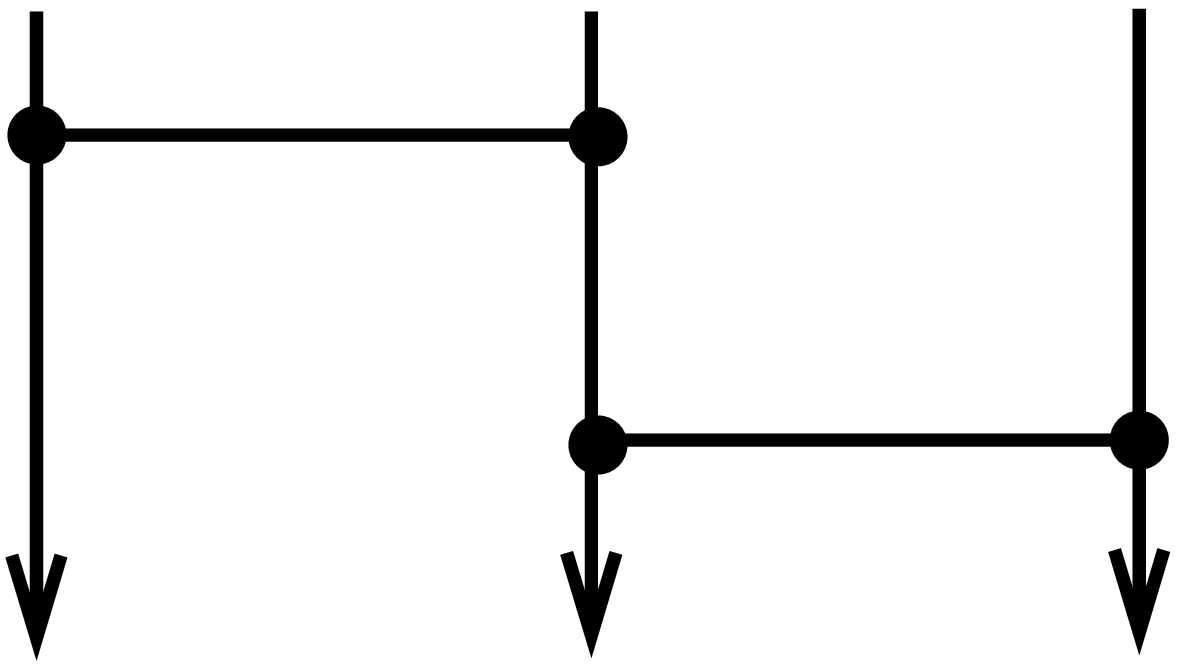}} \end{array}=\begin{array}{c}\scalebox{.15}{\psfig{figure=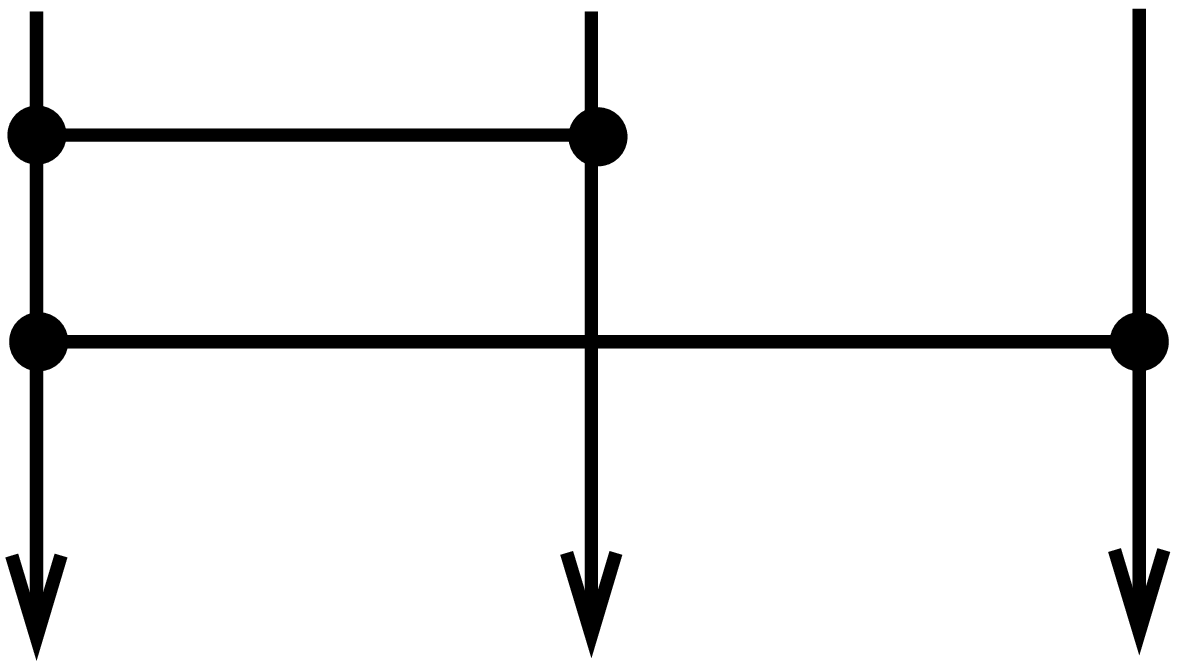}} \end{array}-\begin{array}{c}\scalebox{.15}{\psfig{figure=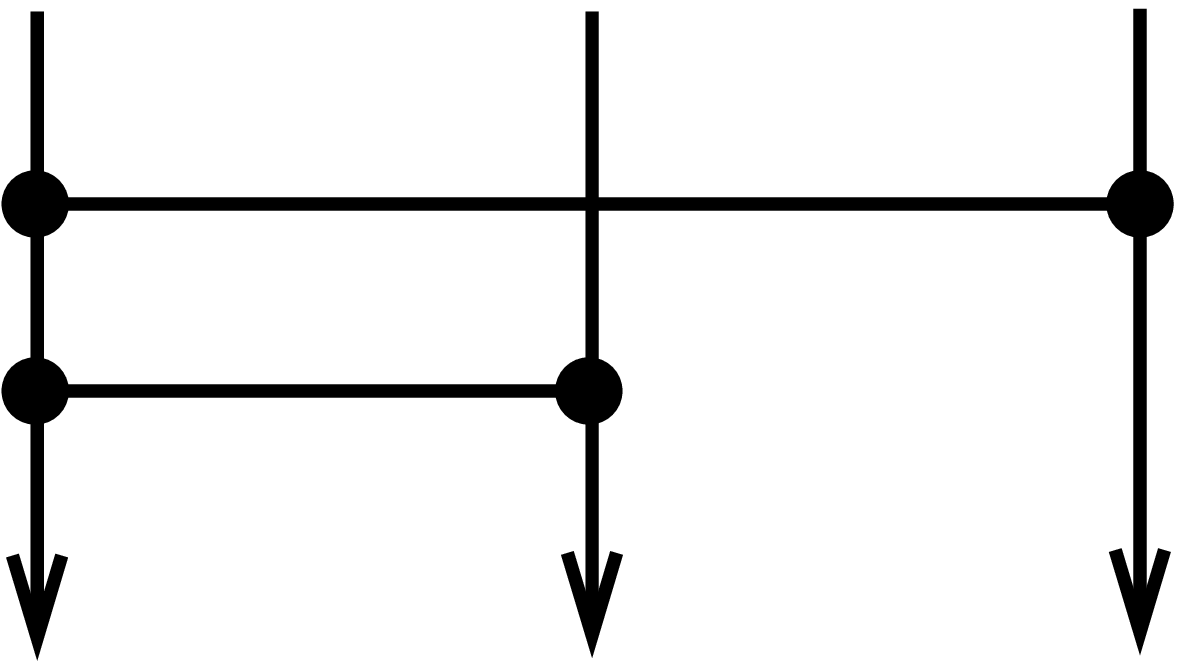}} \end{array}
\]
Define $\mu_n:\overline{\mathscr{F}}_n \to \vec{\mathscr{F}}_n$ to be:
\[
\mu_n \left( \begin{array}{c}\scalebox{.25}{\psfig{figure=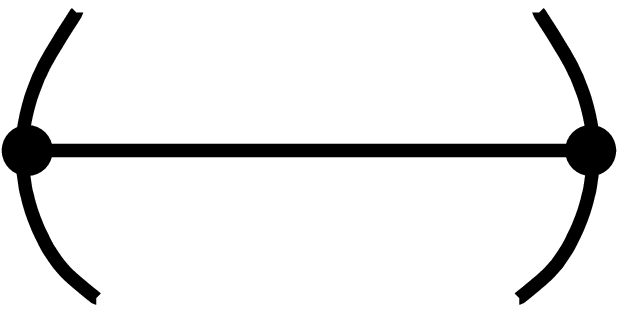}} \end{array} \right)=\begin{array}{c}\scalebox{.25}{\psfig{figure=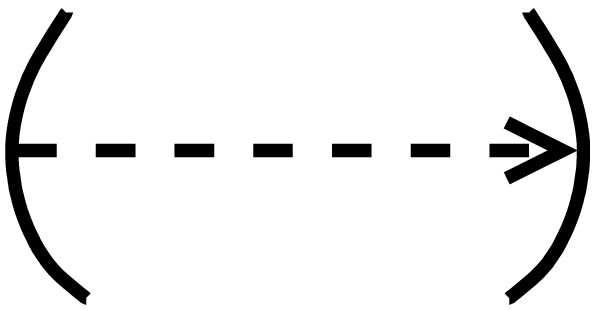}} \end{array}+\begin{array}{c}\scalebox{.25}{\psfig{figure=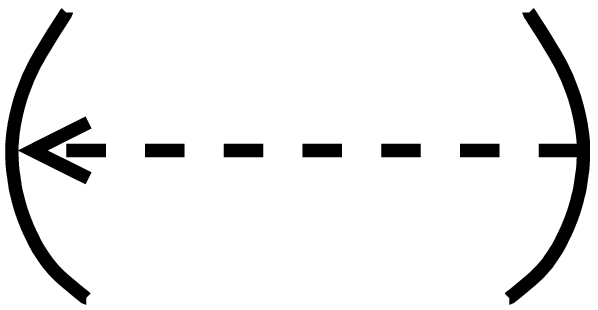}} \end{array}
\]
Here, the sum is over all such resolutions.  In fact, if $F \in \overline{\mathscr{F}}_n$, then $\mu_n(F)$ is a sum of $2^n$ arrow diagrams. The map $\mu_n$ is also called the average map. The following theorem is due to Polyak.
\begin{theorem}[Polyak \cite{Polyak}] The average map $\mu_n:\overline{\mathscr{F}}_n \to \vec{\mathscr{F}}_n$ satisfies the following properties.
\begin{enumerate}
\item $\mu_n\left(\left< \overline{\text{4T}}_n\right>\right)\subset \left< \vec{\text{6T}}_n \right>$
\item $\mu_n\left(\left< \overline{\text{1T}}_n\right>\right)\subset \left< \vec{\text{1T}}_n \right>$
\item The average map descends to a well-defined map:
\[
\mu_n:\frac{\overline{\mathscr{F}}_n}{\left<\overline{\text{1T}}_n,\overline{\text{4T}}_n \right>}\otimes \mathbb{Q} \to \frac{\vec{\mathscr{F}}_n}{\left<\vec{\text{1T}}_n,\vec{\text{6T}}_n \right>}\otimes \mathbb{Q}
\]
\end{enumerate}
\end{theorem}
Define $\mu_n'(D)=\frac{1}{2^n} \cdot \overline{\mu_n(D)}$.  The following diagrams describe this map.
\[
\mu_n':\frac{\overline{\mathscr{F}}_n}{\left<\overline{\text{1T}}_n,\overline{\text{4T}}_n \right>}\otimes \mathbb{Q} \to \frac{\overline{\mathscr{F}}_n}{\left<\overline{\text{1T}}_n,\overline{\text{6T}}_n \right>}\otimes \mathbb{Q}
\]
\[
\xymatrix{ \frac{\overline{\mathscr{F}}_n}{\left<\overline{\text{1T}}_n,\overline{\text{4T}}_n \right>}\otimes \mathbb{Q} \ar[rr]^{2^n \cdot \mu_n'} \ar[dr]_{\mu_n}& & \frac{\overline{\mathscr{F}}_n}{\left<\overline{\text{1T}}_n,\vec{\text{6T}}_n \right>}\otimes \mathbb{Q} \\ & \frac{\vec{\mathscr{F}}_n}{\left<\vec{\text{1T}}_n,\vec{\text{6T}}_n \right>}\otimes \mathbb{Q} \ar[ur]_{\text{Bar}} & }
\]
In fact, for all $D \in \overline{\mathscr{F}}_n$, $\mu_n'(D)=D$.  Hence, $\mu_n'$ is a surjection and $\frac{\overline{\mathscr{F}}_n}{\left<\overline{\text{1T}}_n,\overline{\text{6T}}_n \right>}\otimes \mathbb{Q}$ is a homomorphic image of $\frac{\overline{\mathscr{F}}_n}{\left<\overline{\text{1T}}_n,\overline{\text{4T}}_n \right>}\otimes \mathbb{Q}$. Since $\mu_n'\left(\left<\overline{\text{6T}}_n \right> \right)=\{0\}$, diagrams in $\overline{\mathscr{F}}_n$ satisfy relations $\overline{\text{1T}}_n$, $\overline{\text{6T}}_n$, \emph{and} $\overline{\text{4T}}_n$ in the homomorphic image $\frac{\overline{\mathscr{F}}_n}{\left<\overline{\text{1T}}_n,\overline{\text{6T}}_n \right>}\otimes \mathbb{Q}$.  Now, note that we may write $\overline{\text{6T}}_n$ in the following form:
\[
\underbrace{\begin{array}{c}\scalebox{.1}{\psfig{figure=fourterm1.eps}} \end{array}-\begin{array}{c}\scalebox{.1}{\psfig{figure=fourterm2.eps}} \end{array}-\left( \begin{array}{c}\scalebox{.1}{\psfig{figure=fourterm3.eps}} \end{array}-\begin{array}{c}\scalebox{.1}{\psfig{figure=fourterm4.eps}} \end{array}\right)}_{\overline{\text{4T}}_n}+\underbrace{\begin{array}{c}\scalebox{.1}{\psfig{figure=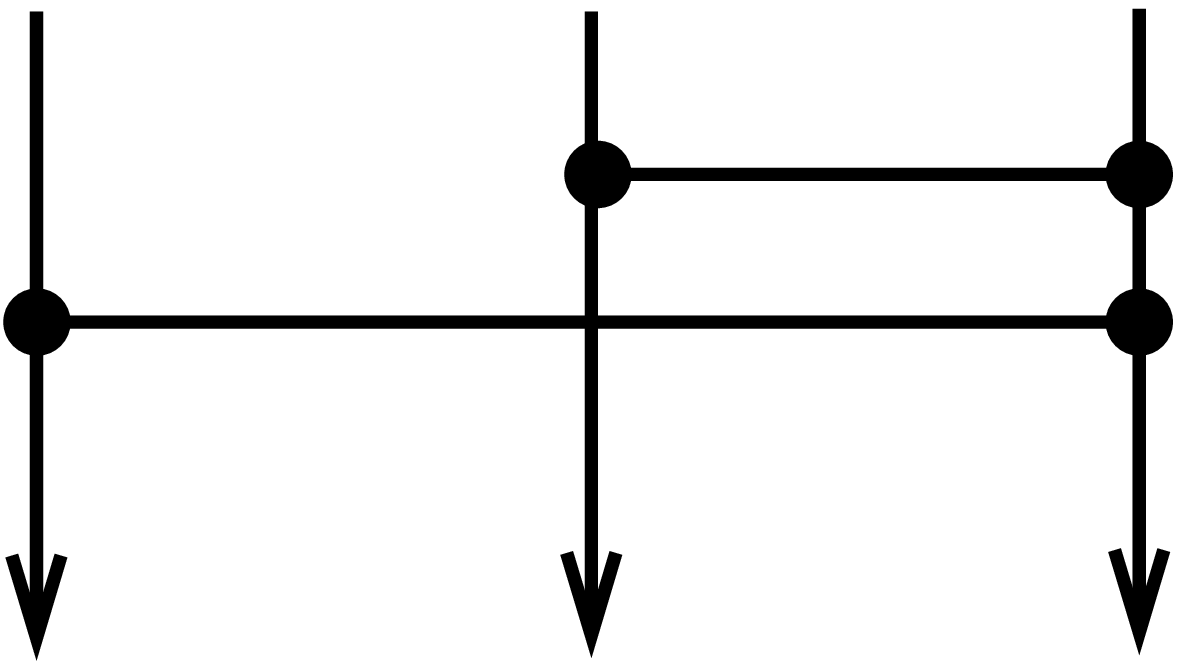}} \end{array}-\begin{array}{c}\scalebox{.1}{\psfig{figure=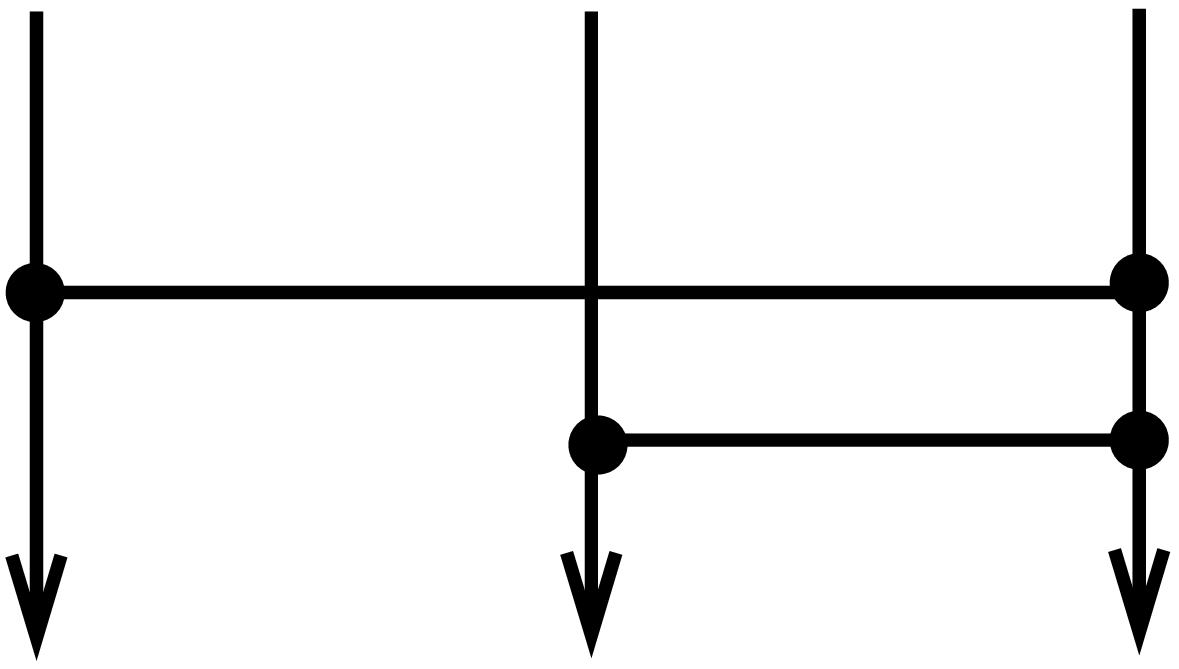}} \end{array}}_{\overline{\text{2T}}_n}=0
\]
Therefore, the two-term relation (see \cite{cdbook}) is also satisfied in the homomorphic image $\frac{\overline{\mathscr{F}}_n}{\left<\overline{\text{1T}}_n,\overline{\text{6T}}_n \right>}\otimes \mathbb{Q}$.
\[
\underline{\overline{\text{2T}}_n:} \,\,\, \begin{array}{c}\scalebox{.1}{\psfig{figure=twoterm1.eps}} \end{array}=\begin{array}{c}\scalebox{.1}{\psfig{figure=twoterm2.eps}} \end{array}
\]
\begin{lemma} If the two term relation is satisfied, then all chord diagrams with $n$ chords are equivalent to the diagram given below:
\newline
\newline
\centerline{
\begin{tabular}{cc}
Virtual Knot Case & Long Virtual Knot Case \\
$\underbrace{\begin{array}{c}\scalebox{.25}{\psfig{figure=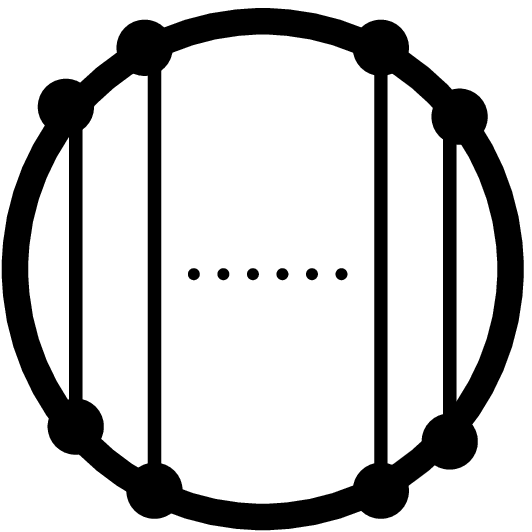}} \end{array}}_{n \text{ chords}} $& $\underbrace{\begin{array}{c}\scalebox{.25}{\psfig{figure=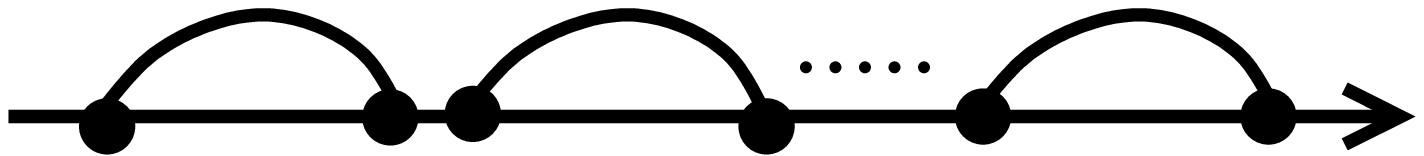}} \end{array}}_{n \text{ chords}}$
\end{tabular}
}
\end{lemma}

Since these diagrams all vanish in the presence of $\overline{\text{1T}_n}$, it follows that $\frac{\overline{\mathscr{F}}_n}{\left<\overline{\text{1T}}_n,\overline{\text{6T}}_n \right>}\otimes \mathbb{Q} \cong \{ 0 \}$.  Hence the dual $\text{Hom}_{\mathbb{Z}} \left( \frac{\overline{\mathscr{F}}_n}{\left<\overline{\text{1T}}_n,\overline{\text{6T}}_n \right>}\otimes \mathbb{Q},\mathbb{Q}\right) \cong \{ 0 \}$ and we conclude that $\pi_n^*$ is an isomorphism.  This implies that $\text{Hom}_{\mathbb{Z}}(\mathscr{V}_n, \mathbb{Q}) \cong \text{Hom}_{\mathbb{Z}}(\mathscr{V}_{n+1}, \mathbb{Q})$ for all $n$.  In the case of virtual knots, we have by Lemma \ref{lowo} that $\text{Hom}_{\mathbb{Z}}(\mathscr{V}_3,\mathbb{Q})=<1>$.  In the case of long virtual knots, $\text{Hom}_{\mathbb{Z}}(\mathscr{V}_2,\mathbb{Q})=<1>$.  This completes the proof of Theorem \ref{bigthm}.

\begin{corollary}[Kauffman\cite{virtkauff},Chrisman\cite{myfirst}] If $v_n$ is the $n$-th Birman coefficient of the Jones-Kauffman polynomial, then $v_n$ is of Kauffman finite-type $\le n$, but not of GPV finite-type $\le m$ for any $m$.
\end{corollary}
\begin{proof} The invariants $v_n$ are invariant under the virtualization move \cite{virtkauff}.
\end{proof}
\bibliographystyle{plain}
\bibliography{bib_vass}

\end{document}